\documentclass[review,onefignum,onetabnum]{siamart171218}

\usepackage{lipsum}
\usepackage{amsfonts}
\usepackage{graphicx}
\usepackage{epstopdf}
\usepackage{algorithmic}
\usepackage{amssymb}
\usepackage[english]{babel}
\usepackage{array}
\usepackage{adjustbox}
\usepackage{relsize}
\usepackage{spverbatim}
\usepackage{listings}
\usepackage{mathrsfs}

\DeclareMathAlphabet{\mathpzc}{OT1}{pzc}{m}{it}

\newcolumntype{L}{>{$}l<{$}}

\ifpdf
  \DeclareGraphicsExtensions{.eps,.pdf,.png,.jpg}
\else
  \DeclareGraphicsExtensions{.eps}
\fi

\newsiamremark{hypothesis}{Hypothesis}
\crefname{hypothesis}{Hypothesis}{Hypotheses}
\newsiamthm{claim}{Claim}

\nolinenumbers

\headers{Bessel Functions of Imaginary Order}{T. M. Dunster}

\title{Uniform Asymptotic Expansions for Bessel Functions of Imaginary Order and their zeros}

\author{T. M. Dunster\thanks{Department of Mathematics and Statistics, San Diego State University, 5500 Campanile Drive, San Diego, CA 92182-7720, USA. 
  (\email{mdunster@sdsu.edu}, \url{https://tmdunster.sdsu.edu}).}
  }

\usepackage{amsopn}

\makeatletter
\newcommand*{\addFileDependency}[1]{
  \typeout{(#1)}
  \@addtofilelist{#1}
  \IfFileExists{#1}{}{\typeout{No file #1.}}
}
\makeatother

\nolinenumbers

\ifpdf
\hypersetup{
  pdftitle={Uniform asymptotic expansions for solutions of differential equations having coalescing turning points},
  pdfauthor={T. M. Dunster}
}
\fi

\begin{document}

\maketitle

\begin{abstract}
Bessel and modified Bessel functions of imaginary order $i\nu$ ($\nu >0$) are studied. Asymptotic expansions are derived as $\nu \to \infty$ that are uniformly valid in unbounded complex domains, with error bounds provided. Coupled with appropriate connection formulas the approximations are uniformly valid for all complex argument. The expansions are of two forms, Liouville-Green type expansions only involving elementary functions, and ones involving Airy functions that are valid at a turning point of the defining differential equation. The new results have coefficients and error bounds that are simpler than in prior expansions, and further are used to construct asymptotic expansions for the zeros of Bessel and modified Bessel functions of large imaginary order, these being uniformly valid without restriction on their size (small or large). 
\end{abstract}



\section{Introduction}
\label{sec:Introduction}

Bessel functions of imaginary order play a crucial role in quantum mechanics, particularly in solving the one-dimensional Schrödinger equation with exponential potentials \cite{Krynytskyi:2021:AEE}. They are also significant in the study of various hydrodynamical models \cite{Palm:1957:OZB}, \cite{Taylor:1931:EVD}. A method for analysing wave-field dispersion relations involving these Bessel functions was detailed in \cite{Chapman:2012:ASD}, with applications in astronomy and oceanography.

The modified Bessel function $K_{i\nu}(x)$ is the unique standard Bessel function that is real for positive argument $x$, and more importantly is recessive at $x = \infty$. In contrast, the Bessel functions $J_{i\nu}(x)$, $Y_{i\nu}(x)$, and the modified Bessel function $I_{i\nu}(x)$ are complex when $\nu$ and $x$ are real and nonzero and are dominant at infinity. $K_{i\nu}(x)$ serves as the kernel of the Kontorovich-Lebedev transform \cite{Palm:1957:OZB}, \cite{Wong:1981:AEK}, \cite{Yakubovich:1996:ITS}. For additional references to other physical applications, see \cite{Dunster:1990:BFP} and \cite{Shi:2009:HEM}.

An algorithm for evaluating $K_{i\nu}(x)$, utilising series, a continued fraction method, and non-oscillating integral representations, was presented in \cite{Gil:2002:EMB}. Computation using Airy uniform asymptotics for $K_{i\nu}(x)$ was studied in \cite{Gil:2003:CMB}, with error bounds for Airy expansions provided by integral methods in \cite{Shi:2010:EBU}. Additionally, in \cite{Shi:2009:HEM} hyperasymptotic expansions for $K_{i\nu}(x)$ for positive $\nu$ and $x$ were obtained employing Hadamard series. 

In \cite{Balogh:1967:AMB} and \cite[Chap. 11, Ex. 10.6]{Olver:1997:ASF} a uniform asymptotic expansion of $K_{i \nu}(z)$ is given which involves Airy functions, and in \cite{Dunster:1990:BFP} similar expansions were derived for other Bessel functions. Asymptotic expansions for the zeros of various Bessel and modified Bessel functions of large imaginary order are given in \cite{Dunster:1990:BFP}, and the $\nu$-zeros of Bessel functions of imaginary order for a fixed positive argument have also been examined, most recently in \cite{Paris:2022:NZK} and \cite{Paris:2022:NZB}.

The modified Bessel functions $z^{1/2}\mathcal{L}_{\pm i\nu}(\nu z)$ ($\mathcal{L}=I,K$) satisfy the normalised modified Bessel differential equation
\begin{equation}
\label{eq00}
\frac{d^{2} w}{dz^{2}}=
\left\{\nu^{2}\frac{z^{2}-1}{z^{2}} 
-\frac{1}{4z^{2}}\right\} w,
\end{equation}
which forms the basis of this paper (cf. \cite[Chap. 11, Eq. (10.01)]{Olver:1997:ASF}). This equation has a regular singularity at $z=0$ with imaginary exponent $i \nu$, an irregular singularity at $z=\infty$, and for large $\nu$ it has turning points at $z= \pm 1$. Here we obtain asymptotic expansions for Bessel and modified Bessel functions for large positive $\nu$ that are uniformly valid in well-defined unbounded complex domains which include points arbitrarily close to $z=0$. We primarily consider the half-plane $|\arg(z)|\leq \frac12 \pi$, with extension to other regions obtained by appropriate connection formulas (see, for example, \cite[Sec. 10.34]{NIST:DLMF}). We shall also obtain simple, uniform asymptotic expansions for their zeros which are more general than currently exist. We illustrate the high accuracy of these approximations with some numerical examples.

The plan of the paper is as follows. In \cref{sec:LGExpansions}, Liouville-Green (LG) expansions that involve the exponential function with an asymptotic expansion in inverse powers of $\nu$ contained in its argument are constructed using results given in \cite{Dunster:2020:LGE}. The advantage of this less usual form is two-fold: firstly, the coefficients in the expansion are generally simpler to evaluate, as are the error bounds. Secondly, the form makes it easier to obtain asymptotic expansions for the zeros of the functions being approximated. In fact, in \cref{sec:LGzeros}, we use our LG solutions to obtain asymptotic expansions for the zeros of any linear combination of $\Re\{J_{i\nu}(\nu x)\}$ and $\Im\{J_{i\nu}(\nu x)\}$, and these are uniformly valid for $0<x<\infty$; in other words, all the real zeros of these functions. We remark that the sequence of these zeros approaches both $x=0$ and $x=\infty$.

As is typical, our LG expansions are not valid at the turning point $z=1$. In \cref{sec:AiryExpansions} expansions that are valid at this point are constructed using the method of \cite{Dunster:2017:COA}. These expansions are similar to the standard form given in \cite[Chap. 11]{Olver:1997:ASF} which employs Airy functions and their derivatives, except that our solutions involve two explicitly given slowly varying functions. They are both expanded as asymptotic expansions using the LG expansions of \cref{sec:LGExpansions}. As a result, the coefficients involved and the associated error bounds are considerably simpler to evaluate than in the standard form. Our expansions are valid in an unbounded complex domain that includes all points in the right half-plane $|\arg(z)| \leq \frac12 \pi$ except $z=1$. Extension to this turning point is then achieved in two ways: either by using Cauchy's integral formula, or by reexpanding the expansions in a standard form where the coefficients of the series are analytic at $z=1$.

The positive zeros of $K_{i \nu}(\nu z)$ lie in $(0,1)$ and are infinite in number, with the singularity at $z=0$ (which recall has an imaginary exponent) being the limit point of this sequence. In \cref{sec:Kzeros} we use the method of \cite{Dunster:2024:AZB} to construct asymptotic expansions as $\nu \to \infty$ for these zeros (as well as a companion modified Bessel function). These are more powerful than existing results, since our new approximations are uniformly valid for all zeros, in particular including those arbitrarily close to $x=0$.

\section{Liouville-Green Expansions}
\label{sec:LGExpansions}

LG and Airy function expansions involve new variables typically denoted by $\xi$ and $\zeta$, respectively (see \cite[Chap. 10, Sec. 2.1 and Chap. 11, Sec. 3.1]{Olver:1997:ASF}). LG expansions involve exponential functions and are not valid at turning points, whereas Airy function approximations are used in regions containing a turning point. In this section we construct the former for the modified Bessel functions $I_{\pm i \nu}(\nu z)$ and $K_{i \nu}(\nu z)$.

On identifying (\ref{eq00}) with \cite[Chap. 10, Eq. (1.01) and Chap. 11, Eq. (3.01]{Olver:1997:ASF} we observe that $u=\nu$ and $f(z)=(z^2-1)/z^2$, and hence $\xi$ and $\zeta$ are given by
\begin{equation}
\label{eq01}
\xi=\frac{2}{3}\zeta^{3/2}
=\int_{1}^{z}\frac{\left(t^2-1\right)^{1/2}}{t}dt
=\left(z^{2}-1\right)^{1/2}
-\mathrm{arcsec}(z).
\end{equation}
Principal branches are taken in (\ref{eq01}) so that the interval $1 \leq z < \infty$ is mapped to $0 \leq \zeta < \infty$ and $0 \leq \xi < \infty$, with $\xi$ being a continuous function in the $z$ plane having a cut along $(-\infty,1]$, and $\zeta$ is an analytic function of $z$ in a region that contains the principal right half-plane excluding $z=0$ (at which $\zeta$ is unbounded). 

For $z=x \in (0,1]$ we have $\zeta \in (-\infty,0]$ such that
\begin{equation}
\label{eq03}
\frac{2}{3}(-\zeta)^{3/2}
=\int_{x}^{1}\frac{\left(1-t^2\right)^{1/2}}{t}dt
=\ln\left\{
\frac{1+(1-x^{2})^{1/2}}{x}
\right\}-\left(1-x^{2}\right)^{1/2},
\end{equation}
and we note that
\begin{equation}
\label{eq01a}
\zeta = 2^{1/3}(z-1) - \tfrac{3}{10}2^{1/3}(z-1)^2
+\mathcal{O}\left\{(z-1)^{3}\right\}
\quad (z \to 1),
\end{equation}
as well as
\begin{equation}
\label{eq02}
\xi=z-\tfrac{1}{2}\pi +\mathcal{O}\left(z^{-1}\right)
\quad  (z \to \infty).
\end{equation}

For $z=x \pm i0$ with $0 < x \leq 1 $ it is readily verified from (\ref{eq01}) that
\begin{equation}
\label{eq04}
\xi
=\mp i \ln\left\{
\frac{1+(1-x^{2})^{1/2}}{x}
\right\} \pm i \left(1-x^{2}\right)^{1/2},
\end{equation}
and as $x \to 0 \pm i0$ we find that $\xi \to \mp i \infty$ such that
\begin{equation}
\label{eq05}
\xi
=\pm i\ln\left(\tfrac{1}{2}x\right) \pm i
+ \mathcal{O}\left(x^{2}\right).
\end{equation}
In \cref{fig:Figzplane,fig:Figxiplane} a map of the first quadrant $0 \leq \arg(z) \leq \frac12 \pi$ to the $\xi$ plane is shown, with corresponding points in the respective planes labeled $\mathsf{A}$ - $\mathsf{G}$.

\begin{figure}
 \centering
 \includegraphics[
 width=0.6\textwidth,keepaspectratio]
 {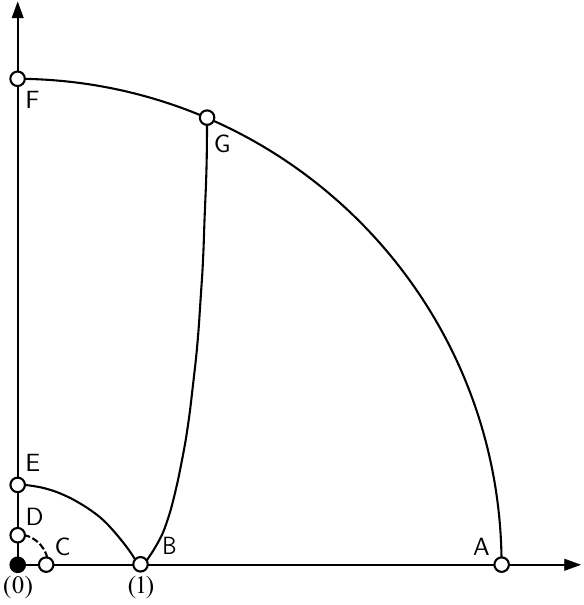}
 \caption{$z$ plane.}
 \label{fig:Figzplane}
\end{figure}

\begin{figure}
 \centering
 \includegraphics[
 width=0.7\textwidth,keepaspectratio]
 {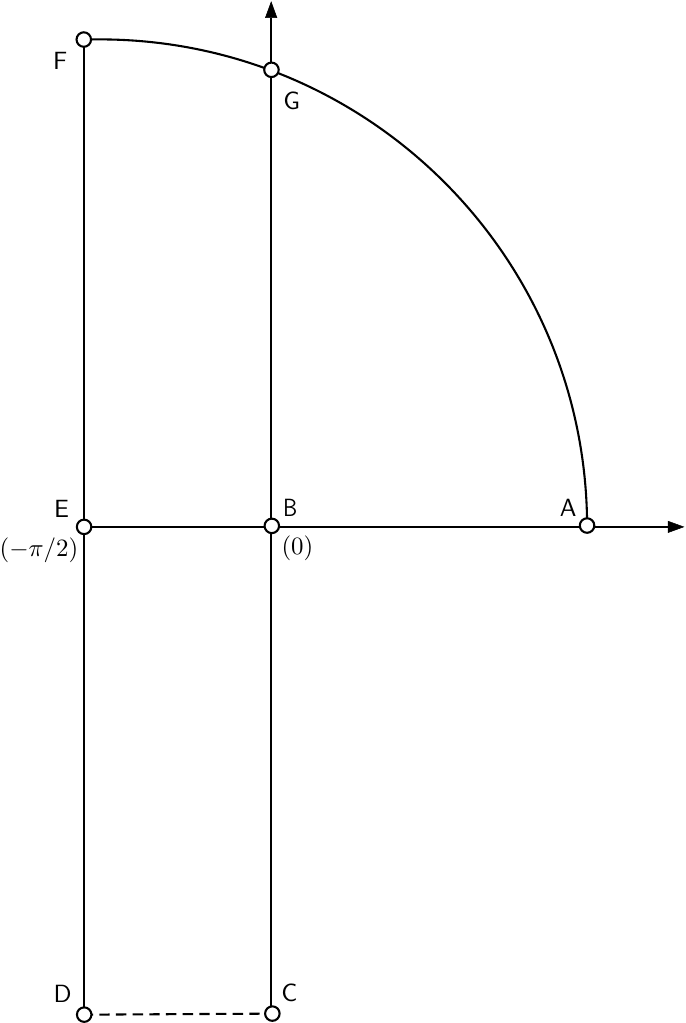}
 \caption{$\xi$ plane.}
 \label{fig:Figxiplane}
\end{figure}

We now apply \cite[Thm. 1.1]{Dunster:2020:LGE} to obtain asymptotic solutions of (\ref{eq00}). In order to construct the coefficients in these expansions it is convenient to work with the variable
\begin{equation}
\label{eq06}
\beta=(z^{2}-1)^{-1/2},
\end{equation}
Here the branch of the square root is positive for $z \in (1,\infty)$ and is continuous in the plane having a cut along $[-1,1]$. Note from (\ref{eq01}) and (\ref{eq06})
\begin{equation}
\label{eq07}
\xi=1/\beta-\mathrm{arccot}(\beta),
\end{equation}
and
\begin{equation}
\label{eq08}
\frac{d\xi}{d \beta}
=-\frac{1}{\beta^{2}(\beta^{2}+1)}.
\end{equation}

Similarly to \cite[Eqs. (5.6) - (5.8)]{Dunster:2021:SEB} we then define a sequence of polynomials via
\begin{equation}
\label{eq09}
\mathrm{E}_{1}(\beta)=\tfrac{1}{24}\beta
\left(5\beta^{2}+3\right),
\end{equation}
\begin{equation}
\label{eq10}
\mathrm{E}_{2}(\beta)=
\tfrac{1}{16}\beta^{2}
\left(5\beta^{2}+1\right)
\left(\beta^{2}+1\right),
\end{equation}
and for $s=2,3,4\ldots$
\begin{equation}
\label{eq11}
\mathrm{E}_{s+1}(\beta) =
\frac{1}{2} \beta^{2} \left(\beta^{2}+1 \right)\mathrm{E}_{s}^{\prime}(\beta)
+\frac{1}{2}\int_{0}^{\beta}
p^{2}\left(p^{2}+1 \right)
\sum\limits_{j=1}^{s-1}
\mathrm{E}_{j}^{\prime}(p)
\mathrm{E}_{s-j}^{\prime}(p) dp.
\end{equation}
Note that by induction one can readily establish that $\mathrm{E}_{2s}(\beta)$ and $\mathrm{E}_{2s+1}(\beta)$ are even and odd, respectively, and all have a factor $\beta^{s}$ (and hence vanish at $\beta=0$). The LG solutions are then given by
\begin{equation}
\label{eq12}
W_{j}(\nu,\zeta) =
\exp \left\{ -\nu\xi +\sum\limits_{s=1}^{n-1}{(-1)^{s}
\frac{\mathrm{E}_{s}(\beta)}{\nu ^{s}}}\right\} 
\left\{ 1+\eta_{n,j}(\nu,z) \right\} 
\quad (j=0,-1),
\end{equation}
and 
\begin{equation}
\label{eq13}
W_{1}(\nu,\zeta) =
\exp \left\{ \nu \xi +\sum\limits_{s=1}^{n-1}
{\frac{\mathrm{E}_{s}(\beta)}{\nu^{s}}}\right\} 
\left\{ 1+\eta_{n,1}(\nu,z) \right\},
\end{equation}
for $n=2,3,4,\ldots$.

Bounds for the error terms $\eta_{n,j}(\nu,z)$ ($j=0,\pm1$) are given as follows. First let
\begin{equation}
\label{eq14}
\mathrm{F}_{s}(z)
=\frac{d \mathrm{E}_{s}(\beta)}{d \xi}
=-\frac{z^2 \mathrm{E}'_{s}(\beta)}{\left(z^2-1\right)^{2}},
\end{equation}
where the prime is differentiation with respect to $\beta$. Let $z=\alpha_{0}=\infty$ ($\arg(z)=0$), and $z=\alpha_{\pm 1}$ correspond to $\xi =-\frac12 \pi -i\infty$ and $\xi = -i\infty$, respectively (see \cref{fig:Figxiplane}). Note that $z=\alpha_{\pm 1}$ corresponds to $z \to 0$ along the positive imaginary and positive real axes, respectively (see \cref{fig:Figzplane}).

Now, on noting that $d\xi/dz=z^{-1}(z^2-1)^{1/2}$, we have from \cite[Thm. 1.1]{Dunster:2020:LGE} for $j=0,\pm 1$ and $n=2,3,4,\ldots$
\begin{equation}
\label{eq15}
\left\vert \eta_{n,j}(\nu,z)\right\vert  
\leq \frac{1}{\nu^{n}} \Phi_{n,j}(\nu,z) 
\exp \left\{\frac{1}{\nu }\Psi_{n,j}(\nu,z) 
+\frac{1}{\nu^{n}} \Phi_{n,j}(\nu,z)\right\},
\end{equation}
where
\begin{multline}
\label{eq16}
\Phi_{n,j}(\nu,z) 
=2\int_{\alpha _{j}}^{z}
{\left\vert {t^{-1}\left(t^2-1\right)^{1/2}  
\mathrm{F}_{n}(t) dt}\right\vert } 
\\
+\sum\limits_{s=1}^{n-1}{\frac{1}{\nu^{s}}
\sum\limits_{k=s}^{n-1}
\int_{\alpha _{j}}^{z}
\left\vert{t^{-1}\left(t^2-1\right)^{1/2}
\mathrm{F}_{k}(t) 
\mathrm{F}_{s+n-k-1}(t) dt}\right\vert},
\end{multline}
and
\begin{equation}
\label{eq17}
\Psi_{n,j}(\nu,z) 
=4\sum\limits_{s=0}^{n-2}\frac{1}{\nu^{s}}
\int_{\alpha _{j}}^{z}
{\left\vert {t^{-1}\left(t^2-1\right)^{1/2}
\mathrm{F}_{s+1}(t) dt}\right\vert }.
\end{equation}

In all integrals the paths (i) consist of a finite chain of $R_{2}$ arcs (as defined in \cite[Chap. 5, \S 3.4]{Olver:1997:ASF}), and (ii) as $v$ passes along the path from $\alpha_{j}$ to $z$  $\Re\{\xi(v)\}$ is monotonic, where $\xi(v)$ is given by (\ref{eq01}) with $z=v$. Let the regions of validity in the $z$ plane be denoted by $Z_{j}$, with $\Xi_{j}$ denoting the corresponding ones in the $\xi$ plane. Then $Z_{j}$ consist of all points $z$ that can be linked to $\alpha_{j}$ by such a path and for which the integrals converge. They do converge at $z=\infty$, but not at the turning points $z=\pm 1$. Then it is seen that $Z_{0}$ and $Z_{1}$ include all points in the first quadrant $0 \leq \arg(z) \leq \frac12 \pi$ except the interval $[0,1]$ for $Z_{0}$ and $z=1$ for $Z_{1}$.

These regions can extend beyond the first quadrant; for example, for $j=-1$ we see from \cref{fig:Figxiplane} that $\Xi_{-1}$ consists of all points to the left of the imaginary axis $\mathsf{C}\mathsf{B}\mathsf{G}$, and also beyond the line $\Im(\xi)=-\frac12 \pi$ subject to the monotonicity condition on $\Re(\xi)$. In particular, in the corresponding $z$ plane all points in the second quadrant $\frac12 \pi \leq \arg(z) \leq \pi$ are included, with the exception of the turning point at $z=-1$. Thus $Z_{-1}$ includes points in the half plane $0 \leq \arg(z) \leq \pi$ that lie to the left of the curve $\mathsf{B}\mathsf{G}$ shown in \cref{fig:Figzplane}, excluding points on this curve and $z=-1$.

It is worth emphasising that $Z_{-1}$ does not include $\Re(z)=\infty$ ($\Re(\xi)=\infty$) and as such $W_{-1}(\nu,\zeta)$, unlike $W_{0}(\nu,\zeta)$, is not recessive as $\Re(z) \to \infty$. However, it is uniquely characterised by its oscillatory behaviour as $z \to 0$, with the same being true for $W_{1}(\nu,\zeta)$.

In order to match the LG solutions with modified Bessel functions we note from \cite[Eqs. 10.25.3, 10.30.1, 10.30.4 and 10.30.5]{NIST:DLMF}
\begin{equation} 
\label{eq18}
K_{i\nu}(\nu z)\sim
\left(\frac{\pi}{2\nu z}\right)^{1/2}e^{-\nu z}
\quad   \left(z \rightarrow \infty, \,
|\arg(z)| \leq \tfrac{3}{2} \pi - \delta\right),
\end{equation}
\begin{equation} 
\label{eq19}
I_{i\nu}(\nu z)\sim
\left(\frac{1}{2 \pi \nu z}\right)^{1/2}e^{\nu z}
\quad   \left(z \rightarrow \infty, \,
|\arg(z)| \leq \tfrac{1}{2} \pi - \delta\right),
\end{equation}
\begin{equation} 
\label{eq20}
I_{-i\nu}(\nu z)\sim
i \left(\frac{1}{2 \pi \nu z}\right)^{1/2}e^{-\nu (z-\pi)}
\quad   \left(z \rightarrow \infty, \,
 \tfrac{1}{2} \pi + \delta \leq \arg(z) 
 \leq \tfrac{3}{2} \pi - \delta\right),
\end{equation}
and in addition 
\begin{equation} 
\label{eq21}
I_{\pm i\nu}(\nu z)\sim
\frac{\left(\tfrac{1}{2} \nu z\right)^{\pm i \nu}}
{\Gamma(1 \pm i\nu)}
\quad   \left(z \rightarrow 0\right).
\end{equation}

Our LG expansions then read as follows.
\begin{theorem}
Let $\xi$, $\beta$ and coefficients $\mathrm{E}_{s}(\beta)$ ($s=1,2,3,\ldots$) be defined by (\ref{eq01}), (\ref{eq06}), and (\ref{eq09}) - (\ref{eq11}). Then for $\nu >0$ and $n=2,3,4,\ldots$
\begin{multline} 
\label{eq22}
K_{i\nu}(\nu z) =
\left(\frac{\pi}{2\nu}\right)^{1/2}
\frac{1}{\left(z^2-1\right)^{1/4}}
\exp \left\{ -\nu\xi-\frac{\pi \nu}{2} +\sum\limits_{s=1}^{n-1}{(-1)^{s}
\frac{\mathrm{E}_{s}(\beta)}{\nu ^{s}}}\right\} 
\\ \times
\left\{ 1+\eta_{n,0}(\nu,z) \right\},
\end{multline}
\begin{multline} 
\label{eq23}
I_{i\nu}(\nu z) =
\left(\frac{1}{2 \pi \nu}\right)^{1/2}
\frac{1}{\left(z^2-1\right)^{1/4}}
\exp \left\{\nu\xi +\frac{\pi \nu}{2}
+\sum\limits_{s=1}^{n-1}
{\frac{\mathrm{E}_{s}(\beta)}{\nu ^{s}}}\right\} 
\\ \times
\frac{ 1+\eta_{n,1}(\nu,z)}{1+\eta_{n,1}(\nu,\infty)},
\end{multline}
and
\begin{multline} 
\label{eq24}
I_{-i\nu}(\nu z) =
i \left(\frac{1}{2 \pi \nu}\right)^{1/2}
\frac{1}{\left(z^2-1\right)^{1/4}}
\exp \left\{ -\nu\xi + \frac{\pi \nu}{2} +\sum\limits_{s=1}^{n-1}{(-1)^{s}
\frac{\mathrm{E}_{s}(\beta)}{\nu ^{s}}}\right\} 
\\ \times
\frac{ 1+\eta_{n,-1}(\nu,z)}{1+\eta_{n,-1}(\nu,\infty e^{\pi i})},
\end{multline}
where $\eta_{n,j}(\nu,z)$ ($j=0,\pm 1$) are bounded by (\ref{eq15}) uniformly for $z \in Z_{j}$ (as described in the paragraph after (\ref{eq17})), and consequently are $\mathcal{O}(\nu^{-n})$ as $\nu \to \infty$ in these regions. Moreover, extension to the conjugates of these domains lying in the lower half $z$ plane follows from employing the above expansions in the reflection relations
\begin{equation} 
\label{eq24a}
I_{\pm i\nu}(\nu \bar{z})=\overline{I_{\mp i\nu}(\nu z)},\;
K_{i\nu}(\nu \bar{z})=\overline{K_{i\nu}(\nu z)}.
\end{equation}
\end{theorem}

\begin{proof}
On matching the recessive solutions at $z=\infty$ in the usual manner, we obtain (\ref{eq22}) by using (\ref{eq18}). Likewise, matching solutions having unique oscillatory behaviour at $z=0$ (see (\ref{eq05}) and (\ref{eq21})) yields (\ref{eq23}) and (\ref{eq24}). In all three the constants of proportionality were determined by using (\ref{eq02}), (\ref{eq12}), (\ref{eq13}), (\ref{eq18}) - (\ref{eq20}), and the behavior of the solutions at $z=\infty$ ($\arg(z)=0$) for (\ref{eq22}) and (\ref{eq23}), and at $z=\infty e^{\pi i}$ (i.e. $z \to \infty$ along $\arg(z)=\pi$) for (\ref{eq24}). Finally, (\ref{eq24a}) follows from the Schwarz reflection principle and that $I_{i\nu}(\nu z)+I_{-i\nu}(\nu z)$, $i\{I_{i\nu}(\nu z)-I_{-i\nu}(\nu z)\}$ and $K_{i\nu}(\nu z)$ are real for positive $\nu$ and $z$.
\end{proof}

Note that in the derivation of (\ref{eq22}) we used (from the error bound (\ref{eq15})) that $\eta_{n,0}(\nu,\infty)=0$, since $\alpha_{0}=\infty$, but there is no reason to suppose that $\eta_{n,1}(\nu,\infty)$ and $\eta_{n,-1}(\nu,\infty e^{\pi i})$ also vanish.

\section{Asymptotics of \texorpdfstring{$J_{i \nu}(t)$}{} for \texorpdfstring{$0<t<\infty$}{}}
\label{sec:LGzeros}
From \cite[Eq. 10.27.6]{NIST:DLMF} for $0<x<\infty$
\begin{equation}
\label{eq25}
J_{i\nu}(\nu x) 
=e^{\pi \nu/2}I_{i\nu}(i\nu x).
\end{equation}
Now (\ref{eq24}) is valid for $z=ix$ with $0<x<\infty$ ($\arg(z)=\frac12 \pi$). For these purely imaginary values of $z$ we have from (\ref{eq04})
\begin{equation}
\label{eq26}
\xi=-\tfrac12 \pi + i\rho,
\end{equation}
where $\rho=\rho(x) \in (-\infty,\infty)$ is given by
\begin{equation}
\label{eq27}
\rho=\left(x^{2}+1\right)^{1/2}
-\ln\left\{\frac{1+(x^{2}+1)^{1/2}}{x}\right\}.
\end{equation}
Note $\rho \to \mp\infty$ as $x \to 0^{+}$ and $x \to \infty$ respectively; more precisely
\begin{equation}
\label{eq37}
\rho=\ln\left(\tfrac12 x \right) +1
+\mathcal{O}(x^2)
\quad (x \to 0^{+}),
\end{equation}
and
\begin{equation}
\label{eq37a}
\rho=x - \frac{1}{2x}
+\mathcal{O}\left(\frac{1}{x^3}\right)
\quad (x \to \infty).
\end{equation}
Also $\rho=0$ for $x=0.6627434193 \cdots$.

Moreover, we have from (\ref{eq06}) and (\ref{eq11})
\begin{equation}
\label{eq28}
\beta = -i \hat{\beta},
\quad
\mathrm{E}_{2s}(\beta)=\hat{\mathrm{E}}_{2s}(\hat{\beta}),
\quad
\mathrm{E}_{2s+1}(\beta)=i\hat{\mathrm{E}}_{2s+1}(\hat{\beta}),
\end{equation}
where the following are all real
\begin{equation}
\label{eq29}
\hat{\beta}=(1+x^2)^{-1/2},
\quad
\hat{\mathrm{E}}_{2s}(\hat{\beta})
=\mathrm{E}_{2s}(-i \hat{\beta}),
\quad
\hat{\mathrm{E}}_{2s+1}(\hat{\beta})
=-i\mathrm{E}_{2s+1}(-i \hat{\beta}).
\end{equation}
From (\ref{eq09}) and (\ref{eq10}) the first two are
\begin{equation}
\label{eq30}
\hat{\mathrm{E}}_{1}(\hat{\beta})
=\tfrac{1}{24}\hat{\beta}
\left(5\hat{\beta}^{2}-3\right),
\end{equation}
and
\begin{equation}
\label{eq31}
\hat{\mathrm{E}}_{2}(\hat{\beta})=
\tfrac{1}{16}\hat{\beta}^{2}
\left(5\hat{\beta}^{2}-1\right)
\left(1-\hat{\beta}^{2}\right).
\end{equation}
Incidentally, from (\ref{eq27}) and (\ref{eq28}) we note that
\begin{equation}
\label{eq32}
\rho=\frac{1}{\hat{\beta}}
+\frac12 \ln\left(\frac{1-\hat{\beta}}
{1+\hat{\beta}}\right).
\end{equation}

Our desired asymptotic expansion is given as follows.

\begin{theorem}
\begin{equation}
\label{eq33}
J_{i\nu}(\nu x) =
R(\nu,x) \, e^{i \Theta(\nu,x)},
\end{equation}
where
\begin{equation}
\label{eq34}
R(\nu,x) =
\left\vert J_{i\nu}(\nu x) \right\vert
\sim 
\frac{1}{(2 \pi \nu)^{1/2}
\left(x^2+1\right)^{1/4}}
\exp \left\{\frac{\nu \pi}{2}
+\sum\limits_{s=1}^{\infty}
{\frac{\hat{\mathrm{E}}_{2s}(\hat{\beta})}
{\nu ^{2s}}}\right\},
\end{equation}
and
\begin{equation}
\label{eq35}
\Theta(\nu,x)
=\arg\left\{J_{i\nu}(\nu x)\right\}
\sim 
\nu \rho(x) -\frac{\pi}{4}
+\sum\limits_{s=0}^{\infty}
{\frac{\hat{\mathrm{E}}_{2s+1}(\hat{\beta})}
{\nu ^{2s+1}}},
\end{equation}
as $\nu \to \infty$ uniformly for $0<x<\infty$. Here $\rho(x)$, $\hat{\beta}$ and the coefficients $\hat{\mathrm{E}}_{s}(\hat{\beta})$ are defined by (\ref{eq09}) - (\ref{eq11}), (\ref{eq27}) and (\ref{eq29}).
\end{theorem}

\begin{proof}
The result follows from (\ref{eq23}), (\ref{eq25}), (\ref{eq26}), (\ref{eq28}) and (\ref{eq29}). That the constant on the RHS of (\ref{eq35}) is $-\frac14 \pi$ as opposed to $(2p-\frac14 )\pi$ for some non-zero integer $p$ can be confirmed from \cite[Eqs. 5.11.1 and 10.7.3]{NIST:DLMF} that
\begin{equation}
\label{eq36}
\arg\left\{J_{i\nu}(\nu x)\right\}
=\nu \ln\left(\tfrac12 x \right) + \nu - \tfrac14 \pi
+\mathcal{O}(\nu x^2)+\mathcal{O}(\nu^{-1}),
\end{equation}
as $\nu x \to 0^{+}$ and $\nu \to \infty$, which is in accord with (\ref{eq37}) and (\ref{eq35}).
\end{proof}

Now taking the first four terms in the expansion (\ref{eq34}) we have the approximation 
\begin{equation}
\label{eq36c}
(2 \pi \nu)^{1/2}
e^{-\nu \pi/2}\left(x^2+1\right)^{1/4}
\left\vert J_{i\nu}(\nu x) \right\vert
\approx
\exp \left\{\sum\limits_{s=1}^{4}
{\frac{\hat{\mathrm{E}}_{2s}(\hat{\beta})}
{\nu ^{2s}}}\right\}.
\end{equation}
To check the accuracy of this, the function
\begin{equation}
\label{eq36d}
\Omega_{1}(\nu,x)=\log_{10}\left| (2 \pi \nu)^{1/2}
e^{-\nu \pi/2}\left(x^2+1\right)^{1/4}
\left\vert J_{i\nu}(\nu x) \right\vert
-\exp \left\{\sum\limits_{s=1}^{4}
{\frac{\hat{\mathrm{E}}_{2s}(\hat{\beta})}
{\nu ^{2s}}}\right\} \right|
\end{equation}
for $\nu=10$ is depicted in \cref{fig:absJ}, and we see that the approximation (\ref{eq36c}) agrees to the exact value to about 10 decimal places or better for $x>0$. We used Maple to calculate the values of $J_{10i}(10x)$ in (\ref{eq36d}) (and similarly below in other error bound computations involving Bessel and modified Bessel functions).

\begin{figure}
 \centering
 \includegraphics[
 width=0.7\textwidth,keepaspectratio]{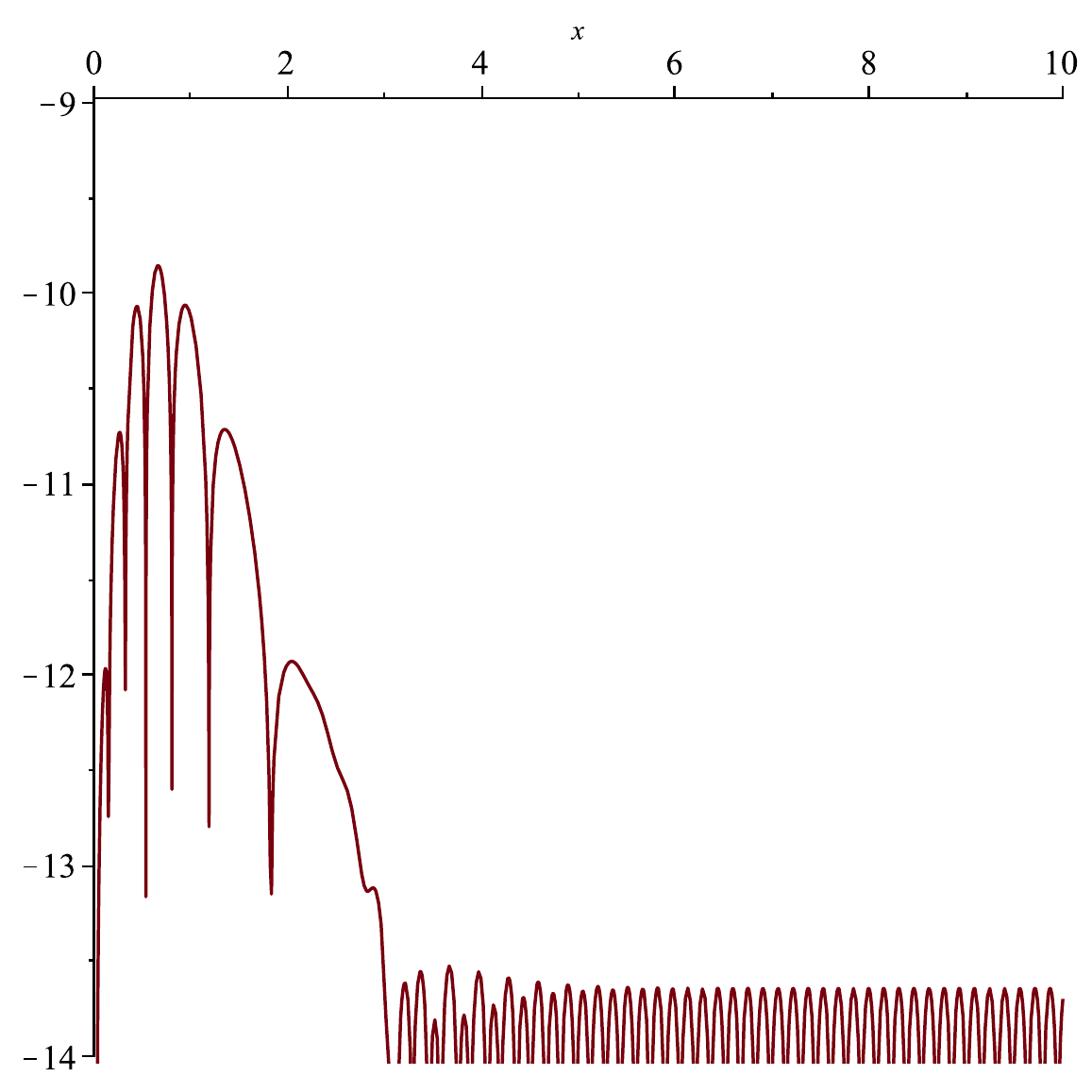}
 \caption{Graph of $\Omega_{1}(10,x)$.}
 \label{fig:absJ}
\end{figure}

Similarly, from (\ref{eq35}) it follows that 
\begin{equation}
\label{eq36a}
e^{-i(\nu\rho-\frac{1}{4}\pi)}
\frac{J_{i\nu}(\nu x)}
{\left| J_{i\nu}(\nu x) \right|} \approx
\exp \left\{i\sum\limits_{s=0}^{4}
{\frac{\hat{\mathrm{E}}_{2s+1}(\hat{\beta})}
{\nu ^{2s+1}}}\right\}.
\end{equation}
The precision of this approximation is illustrated in \cref{fig:argJ}, in which the graph of $\Omega_{2}(\nu,x)$ is shown for $\nu=10$, where
\begin{equation}
\label{eq36b}
\Omega_{2}(\nu,x)=\log_{10}\left|
e^{-i(\nu\rho-\frac{1}{4}\pi)}
\frac{J_{i\nu}(\nu x)}
{\left| J_{i\nu}(\nu x) \right|}-
\exp \left\{i\sum\limits_{s=0}^{4}
{\frac{\hat{\mathrm{E}}_{2s+1}(\hat{\beta})}
{\nu ^{2s+1}}}\right\}\right|.
\end{equation}

\begin{figure}
 \centering
 \includegraphics[
 width=0.7\textwidth,keepaspectratio]{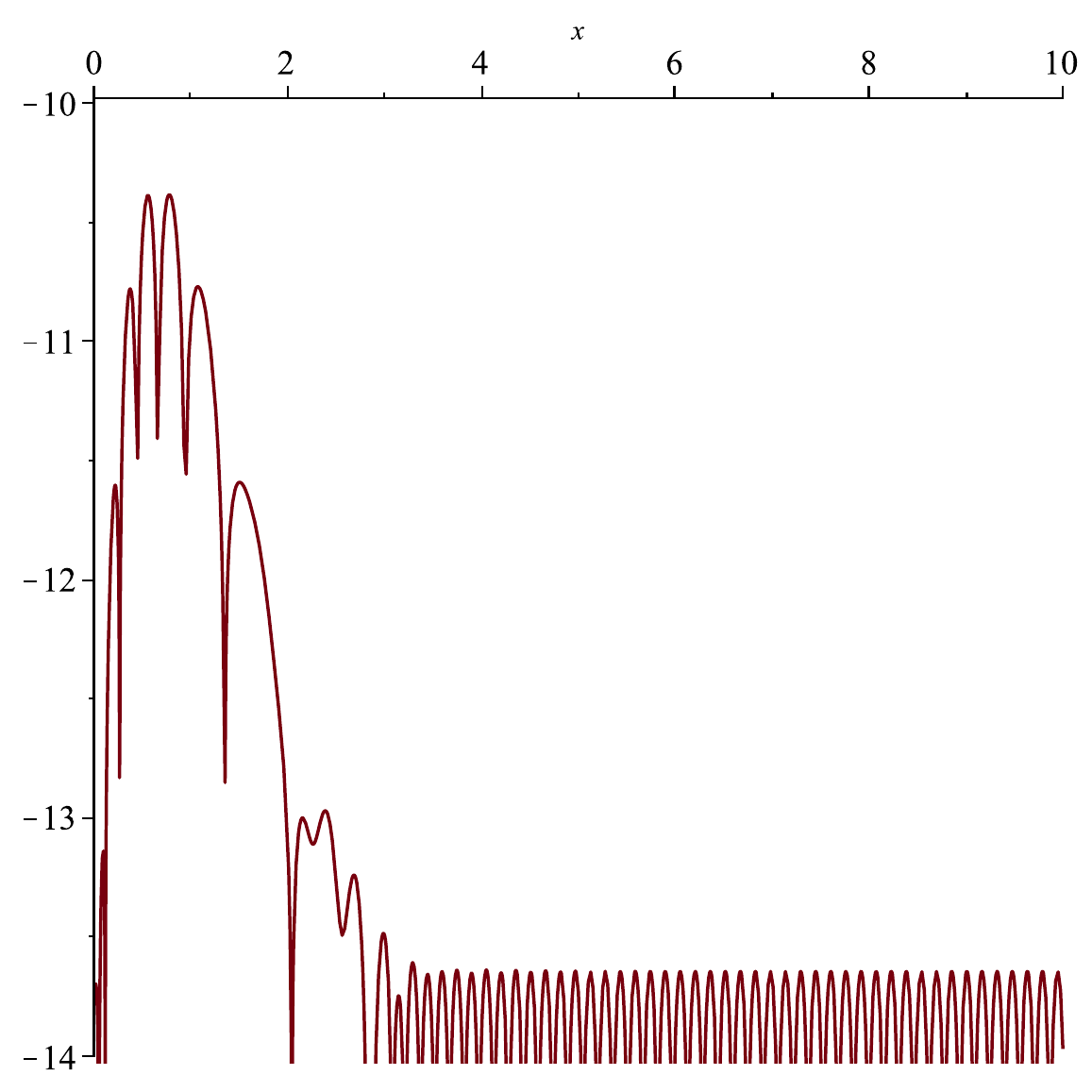}
 \caption{Graph of $\Omega_{2}(10,x)$.}
 \label{fig:argJ}
\end{figure}

The real zeros of $K_{i\nu}(t)$ will be considered in the next section, since these require Airy function expansions. Here we consider the positive zeros of $\Re\{e^{-ir \pi}J_{i\nu}(t)\}$ for $r \in [0,\frac12]$. There are an infinite number of these that approach both $x=0$ and $x=\infty$. From (\ref{eq35}) observe that they are given by $t=\nu x$ where $x$ depends on $\nu$, $m$ and $r$ and satisfies the asymptotic relation
\begin{equation}
\label{eq38}
\nu \rho(x)+\sum\limits_{s=0}^{\infty}
{\frac{\hat{\mathrm{E}}_{2s+1}(\hat{\beta})}
{\nu ^{2s+1}}}
\sim
\left(m+r -\tfrac{1}{4}\right)\pi
\quad (\nu \to \infty, \, m=0,\pm 1,\pm 2,\ldots).
\end{equation}
Let us label these zeros by $t=\nu x_{\nu,m}(r)$, and accordingly
\begin{equation}
\label{eq38a}
\Re\left\{e^{-ir \pi}J_{i\nu}
\left(\nu x_{\nu,m}(r)\right)\right\} =0
\quad (r\in [0,\tfrac12], \, m=0,\pm 1,\pm 2,\ldots).
\end{equation}
Then from (\ref{eq38}) we see that for $\nu>0$ fixed $x_{\nu,m}(r) \to 0^{+}$ as $m \to -\infty$ and $x_{\nu,m}(r) \to \infty$ as $m \to \infty$, since $\rho \to \mp\infty$ as $x \to 0^{+}$ and $x \to \infty$ respectively (see (\ref{eq37}), (\ref{eq37a})), and recalling that the coefficients $\hat{\mathrm{E}}_{s}(\hat{\beta})$ are bounded for $x \in (0,\infty)$.

Note $r=0$ and $r=\frac12$ give the zeros of $\Re\{J_{i\nu}(t)\}$ and $\Im\{J_{i\nu}(t)\}$, respectively, or equivalently the zeros of $\Im\{Y_{i\nu}(t)\}$ and $\Re\{Y_{i\nu}(t)\}$, respectively, since for positive $t$ and $\nu$
\begin{equation}
\label{eq38b}
\Re\{Y_{i\nu}(t)\}
=\frac{\sinh(\nu \pi)}{\cosh(\nu \pi)-1}
\Im\{J_{i\nu}(t)\},
\end{equation}
and
\begin{equation}
\label{eq38c}
\Im\{Y_{i\nu}(t)\}
=-\frac{\sinh(\nu \pi)}{\cosh(\nu \pi)+1}
\Re\{J_{i\nu}(t)\},
\end{equation}
(see \cite[Eq. 10.4.5]{NIST:DLMF}). These also correspond to the positive zeros of the Bessel functions $F_{i \nu}(t)$ and $G_{i \nu}(t)$, respectively, as defined by \cite[Eqs. (3.4a,b)]{Dunster:1990:BFP}. The notation here for these zeros differs from the one used in that paper.

From (\ref{eq38}) the asymptotic expansions hold in the form
\begin{equation}
\label{eq39}
x_{\nu,m}(r)
\sim \sum_{s=0}^{\infty}\frac{p_{m,s}(r)}{\nu^{2s}},
\end{equation}
as $\nu \to \infty$, uniformly for $m=0,\pm 1,\pm 2,\ldots$. On inserting (\ref{eq39}) into (\ref{eq35}), expanding in inverse powers of $\nu$, and equating coefficients of like powers, we recursively can determine the coefficients $p_{m,s}(r)$. A general formula for this recursion is possible, similar to that derived for the Bessel zeros in \cite{Dunster:2024:AZB}, but we do not pursue this here and shall simply record the first five terms. The first is given by
\begin{equation}
\label{eq40}
p_{m,0}(r)= \rho^{-1}(M),
\end{equation}
where
\begin{equation}
\label{eq41}
M=M(\nu,m,r)=\nu^{-1}\left(m+r -\tfrac{1}{4}\right)\pi,
\end{equation}
with $\rho^{-1}(\cdot)$ being the inverse of $\rho(x)$. For each $\nu \in (0,\infty)$, $r \in [0,\frac12]$ and $m \in \mathbb{Z}$ this inverse yields a unique value for $p_{m,0}(r) \in (0,\infty)$ since $\rho(x)$ increases monotonically from $-\infty$ to $\infty$ for $0<x<\infty$.

As $M \to \infty$ we see from (\ref{eq37a}), (\ref{eq38}) and  (\ref{eq39}) that $p_{m,0}(r) \to \infty$ such that
\begin{equation}
\label{eq42}
p_{m,0}(r)= 
M+\frac{1}{2M}
-\frac{7}{24 M^3}
+ \mathcal{O}\left(\frac{1}{M^5}\right).
\end{equation}
As $M \to -\infty$ one finds from (\ref{eq37}), (\ref{eq38}) and  (\ref{eq39}) that $p_{m,0}(r) \to 0^{+}$ such that
\begin{equation}
\label{eq44}
p_{m,0}(r)= 
2e^{M-1}
-2 e^{3M-3}
+\mathcal{O}\left(e^{5M}\right).
\end{equation}

In a similar manner it can be verified that the subsequent coefficients are of the form
\begin{equation}
\label{eq45}
p_{m,s}(r)= q_{s}(p_{m,0}(r))
\quad (s=0,1,2,\ldots),
\end{equation}
where $q_{0}(x)=x$, and the others are rational functions. The next four are found to be given by
\begin{equation}
\label{eq46}
q_{1}(x)=\frac{x\left(3x^2-2\right)}
{24\left(x^2+1\right)^2},
\end{equation}
\begin{equation}
\label{eq47}
q_{2}(x)=-\frac{x
\left(465 x^6 - 4119 x^4 + 1812 x^2 - 4\right)}
{5760\left(x^2+1\right)^5},
\end{equation}
\begin{multline}  
\label{eq48}
q_{3}(x)=\frac{x}
{2903040\left(x^2+1\right)^8}
\left(714231 x^{10} - 19038132 x^8 
+ 46671831 x^6
\right.
\\
\left.
- 19043730 x^4 
+ 910164 x^2 - 1912 \right),
\end{multline}
and
\begin{multline}  
\label{eq49}
q_{4}(x)=-\frac{x}
{1393459200\left(x^2+1\right)^{11}}
\left(2542280985 x^{14} 
- 138922188885 x^{12}
\right.
\\
+ 846638961795 x^{10} - 1239519604671 x^8 + 493158930936 x^6 
\\
\left.
- 45022408056 x^4 + 452367216 x^2 + 742544 \right).
\end{multline}
Graphs of $q_{s}(x)/x$ for $s=1,2,3,4$ are shown in \cref{fig:qplots}, illustrating their relative magnitudes to the leading term $q_{0}(x)=x$.

\begin{figure}
 \centering
 \includegraphics[
 width=0.7\textwidth,keepaspectratio]{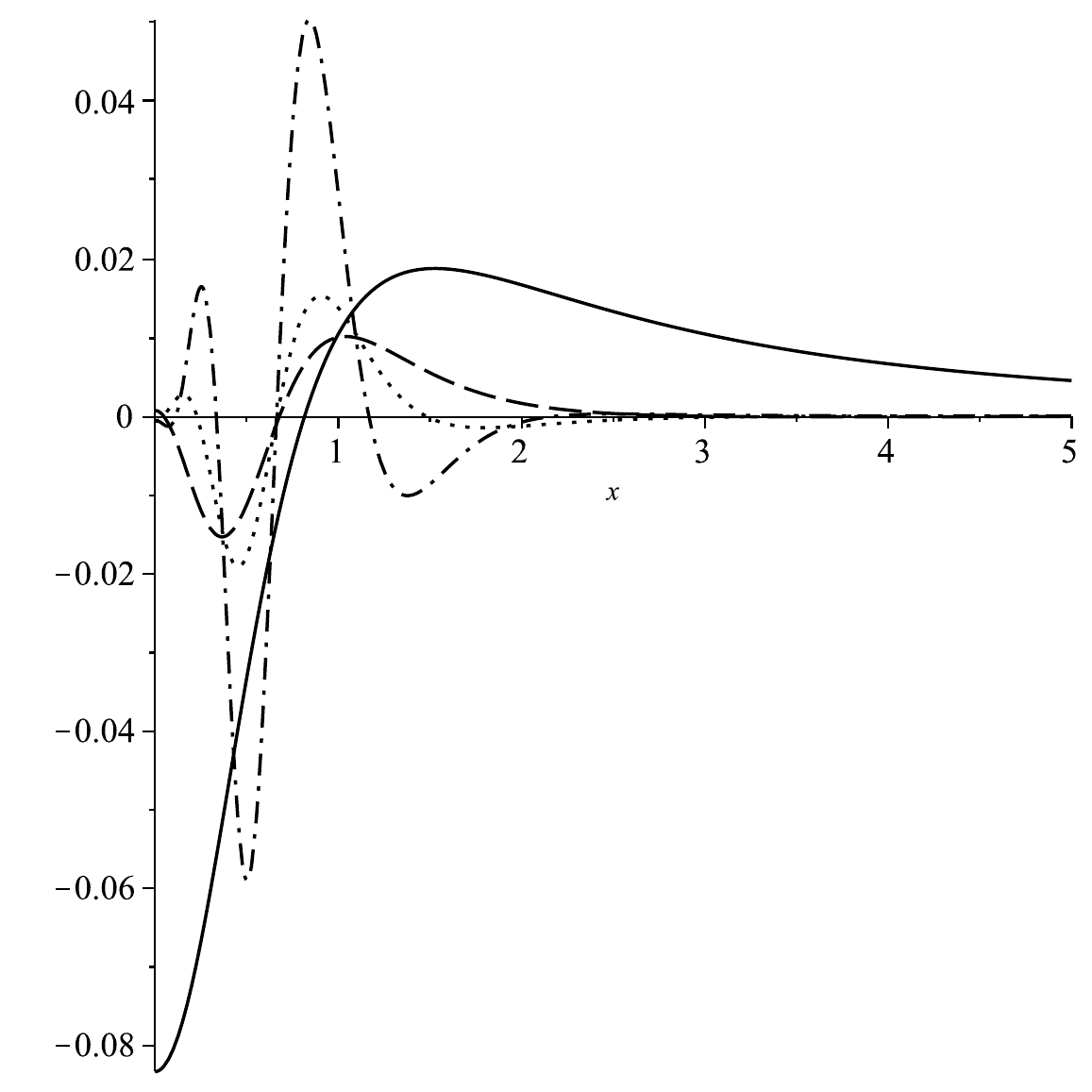}
 \caption{Graphs of $q_{s}(x)/x$ for $s=1$ (solid), $s=2$ (dashed), $s=3$ (dotted) and $s=4$ (dash-dotted).}
 \label{fig:qplots}
\end{figure}

To illustrate the uniform accuracy of the expansion (\ref{eq39}) we compute the truncated expansion
\begin{equation}
\label{eq50}
j_{4}^{(0)}(\nu,m) =
\nu \sum_{s=0}^{4}
\frac{p_{m,s}(0)}{\nu^{2s}},
\end{equation}
which is our approximation for $t=\nu x_{\nu,m}(0)$, the $m$th zero of $\Re\{J_{i\nu}(t)\}$ (since $r=0$). Thus let $\delta_{\nu,m}^{(0)}$ be the relative error in our approximation to the exact zero $\nu x_{\nu,m}(0)$ of $\Re\{J_{i\nu}( t)\}$, i.e.
\begin{equation}
\label{eq51}
j_{4}^{(0)}(\nu,m)=\nu x_{\nu,m}(0)
\left(1+\delta_{\nu,m}^{(0)}\right).
\end{equation}
Rather than computing $\delta_{\nu,m}^{(0)}$ exactly for many values of $m$ we can instead more simply estimate it by evaluating $\Delta_{\nu}^{(0)}(j_{4}^{(0)}(\nu,m))$ where
\begin{equation}
\label{eq52}
\Delta_{\nu}^{(0)}(t)
= \frac{\Re \{J_{i\nu}(t)\}}{t \Re \{J'_{i\nu}(t)\}}.
\end{equation}
In computing this it is helpful to use from \cite[Eq. 10.6.2]{NIST:DLMF} that for real $\nu$ and positive $t$  
\begin{equation}
\label{eq53}
\Re \{J'_{i\nu}(t)\}=-\Re\{J_{1+i\nu}(t)\}
-(\nu/t)\Im\{J_{i\nu}(t)\}.
\end{equation}

\begin{figure}
 \centering
 \includegraphics[
 width=0.7\textwidth,keepaspectratio]{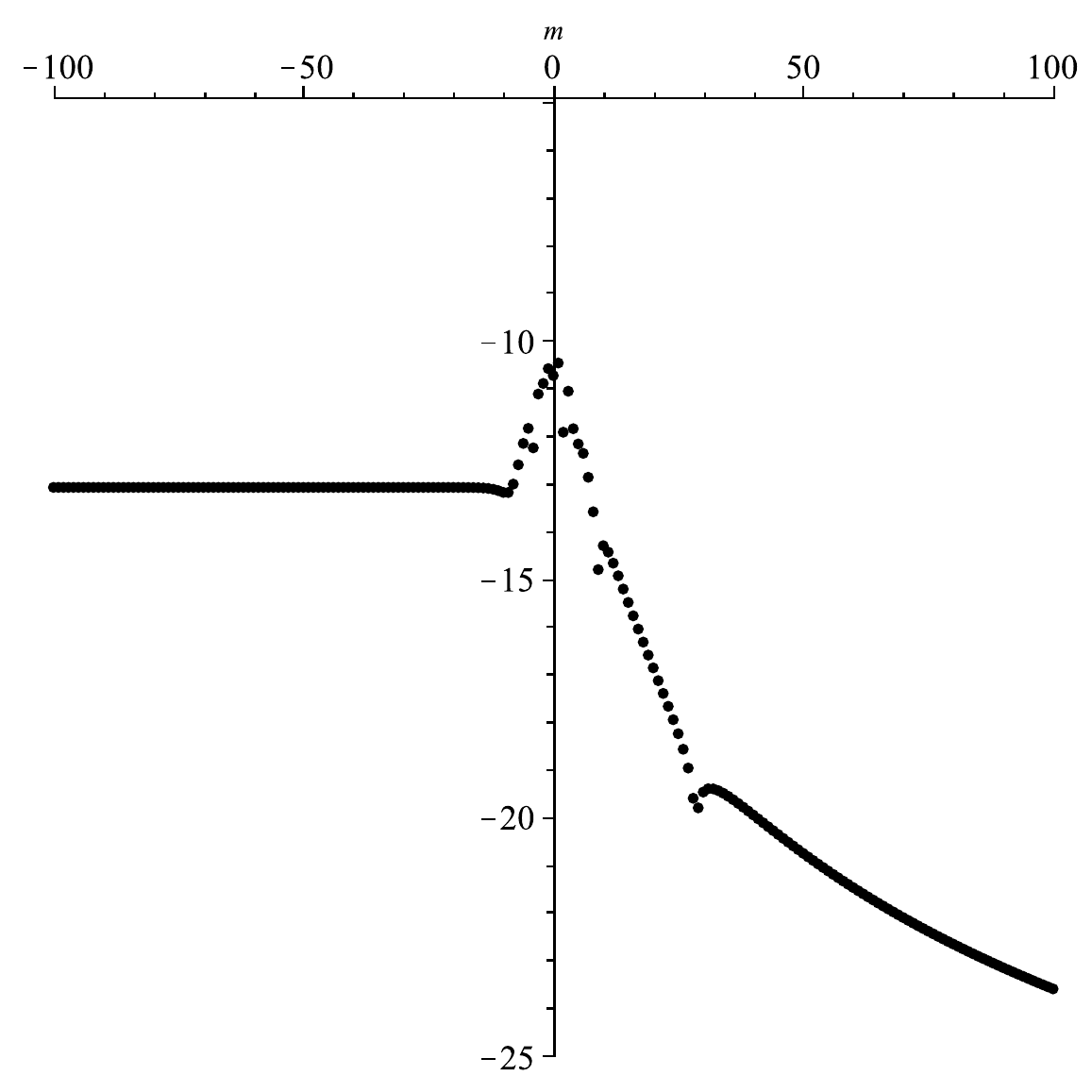}
 \caption{Plots of $\log_{10}|\Delta_{10}^{(0)}(j_{4}^{(0)}(10,m))|$ for $m =-100, -99 , \dots 99,100$.}
 \label{fig:LGzeros}
\end{figure}

To see why $\Delta_{\nu}^{(0)}(j_{4}^{(0)}(\nu,m))
\approx \delta_{\nu,m}^{(0)}$ first note $\Re\{J_{i\nu}(\nu x_{\nu,m}(0))\}=0$, and hence from (\ref{eq52}) $\Delta_{\nu}^{(0)}(\nu x_{\nu,m}(0))=0$. Thus by Taylor's theorem
\begin{equation}
\label{eq54}
\Delta_{\nu}^{(0)}(j_{4}^{(0)}(\nu,m))=
\Delta_{\nu}^{(0)}\left(\nu x_{\nu,m}
(0)(1+\delta_{\nu,m}^{(0)})\right)
= \delta_{\nu,m}^{(0)}
+ \mathcal{O}\left(\{\delta_{\nu,m}^{(0)}\}^{2}\right),
\end{equation}
which confirms our assertion. For example, for $\nu=10$ and $m=\pm 20$ we obtain $j_{4}^{(0)}(10,-20)=0.0012691448 \cdots$ and $j_{4}^{(0)}(10,20)=6.2842314194 \cdots$, from which
\begin{equation}
\label{eq56}
\Delta_{10}^{(0)}(j_{4}^{(0)}(10,-20))
= \mathbf{8.120851953268}283\cdots\times 10^{-14},
\end{equation}
and
\begin{equation}
\label{eq57}
\Delta_{10}^{(0)}(j_{4}^{(0)}(10,20))
= \mathbf{1.33246199579953564}633\cdots\times 10^{-17},
\end{equation}
with the digits in bold agreeing with the exact relative error. The latter values were found in Maple with Digits set to 40 and numerically solving  for small $\epsilon$ the equations (see (\ref{eq38a}) and (\ref{eq51}))
\begin{equation}
\label{eq58}
\Re\left\{J_{10i}\left(j_{4}^{(0)}(10,\pm 20)
(1+\epsilon)^{-1}\right) \right\}= 0.
\end{equation}

A graph of the values of $\log_{10}|\Delta_{\nu}^{(0)}(j_{4}^{(0)}(\nu,m))|$ for integer $m \in [-100,100]$ is shown in \cref{fig:LGzeros} for $\nu=10$. We observe uniformly high accuracy in our approximations for all these values of $m$, with more than 10 digits throughout, and significantly better as $m$ increases through positive values.

\section{Airy Expansions}
\label{sec:AiryExpansions}

Here we obtain Airy function expansions for all standard solutions of (\ref{eq01}) which are valid at the turning point $z=1$, and which involve coefficients and error bounds that are simpler to evaluate than previous results. In doing so the following Airy function connection formula (\cite[Eq. 9.2.12]{NIST:DLMF}) plays an important role
\begin{equation}
\label{eq61}
\mathrm{Ai}(z) 
=e^{\pi i/3}\mathrm{Ai}_{1}(z) 
+e^{-\pi i/3}\mathrm{Ai}_{-1}(z),
\end{equation}
where $\mathrm{Ai}_{j}(z) =\mathrm{Ai}(z e^{-2\pi ij/3})$ ($j=0,\pm 1$). We also shall use (\cite[Eq. 9.2.10]{NIST:DLMF})
\begin{equation}
\label{eq62}
\mathrm{Bi}(z) 
=e^{-\pi i/6}\mathrm{Ai}_{1}(z) 
+e^{\pi i/6}\mathrm{Ai}_{-1}(z).
\end{equation}

Our results employ non-standard asymptotic expansions for the Airy functions where the series appear in the argument of an exponential, as opposed to the usual Poincar\'e forms where the asymptotic series multiplies the leading exponential term (see for example \cite[Sec. 9.7]{NIST:DLMF}). 

These expansions are given as follows. Firstly let $\xi=\frac{2}{3}\zeta^{3/2}$, where in what follows both are general complex-valued variables, but later will be the specific functions of $z$ given by (\ref{eq01}). Then from \cite[Thm. 2.4]{Dunster:2021:SEB} the appropriate Airy function expansions are given by
\begin{equation}
\label{eq63}
\mathrm{Ai}\left(\nu^{2/3}\zeta\right) 
=\frac{1}{2\pi ^{1/2}\nu^{1/6}\zeta^{1/4}}
\exp \left\{ -\nu\xi +\sum\limits_{s=1}^{n-1}{(-1) ^{s}
\frac{a_{s}}{s \nu^{s}\xi ^{s}}}\right\} 
\left\{ 1+\eta_{n}^{(0)}(\nu,\xi)\right\},
\end{equation}
and 
\begin{equation}
\label{eq64}
\mathrm{Ai}^{\prime }\left( \nu^{2/3}\zeta\right)
=-\frac{\nu^{1/6}\zeta ^{1/4}}
{2\pi ^{1/2}}\exp \left\{-\nu\xi
+\sum\limits_{s=1}^{n-1}{(-1)^{s}
\frac{\tilde{a}_{s}}{s \nu^{s}\xi ^{s}}}\right\} 
\left\{ 1+\tilde{\eta}_{n}^{(0)}(\nu,\xi) \right\},
\end{equation}
where $\eta_{n}^{(0)}(\nu,\xi)$ and $\tilde{\eta}_{n}^{(0) }(\nu,\xi)$ are $\mathcal{O}\{(\nu\xi)^{-n}\}$ ($n=2,3,4,\ldots$) as $\nu \xi \to \infty$, uniformly for $|\arg(\zeta)|\leq \frac{2}{3}\pi $ (or equivalently $|\arg(\xi) | \leq \pi $). These regions are not maximal, but suffice for our purposes (and similarly for (\ref{eq66}) - (\ref{eq69}) below).

Here $a_{1}=a_{2}=\frac{5}{72}$, $\tilde{a}_{1}=\tilde{a}_{2}=-\frac{7}{72}$, and subsequent terms $a_{s}$ and $\tilde{{a}}_{s}$ ($s=3,4,5,\ldots $) both satisfy the  recursion formula
\begin{equation}
\label{eq65}
b_{s+1}=\frac{1}{2}\left(s+1\right) b_{s}+\frac{1}{2}
\sum\limits_{j=1}^{s-1}{b_{j}b_{s-j}} 
\quad (s=2,3,4,\ldots).
\end{equation}

Similarly, for $j=1$ and $n=2,3,4,\ldots$,
\begin{equation}
\label{eq66}
\mathrm{Ai}_{1}\left( \nu^{2/3}\zeta \right) 
=\frac{e^{\pi i/6}}{2\pi^{1/2} \nu^{1/6}
\zeta^{1/4}}\exp \left\{ \nu\xi 
+\sum\limits_{s=1}^{n-1}\frac{a_{s}}
{s \nu^{s}\xi^{s}}\right\} 
\left\{ 1+\eta_{n}^{(1)}(\nu,\xi) \right\},
\end{equation}
and 
\begin{equation}
\label{eq67}
\mathrm{Ai}_{1}^{\prime }
\left( \nu^{2/3}\zeta \right) 
=\frac{e^{\pi i/6} \nu^{1/6}\zeta ^{1/4}}{2\pi ^{1/2}}
\exp \left\{ \nu\xi+\sum\limits_{s=1}^{n-1}\frac{\tilde{a}_{s}}{s \nu^{s}\xi ^{s}}\right\}
\left\{ 1+\tilde{\eta}_{n}^{(1) }(\nu,\xi) \right\},
\end{equation}
where the error terms are also $\mathcal{O}\{(\nu\xi)^{-n}\}$ as $\nu \xi \to \infty$, uniformly for a sector that contains $0 \leq \arg(\xi) \leq \frac12 \pi$.

For $j=-1$ and $n=2,3,4,\ldots$ we have
\begin{multline}
\label{eq68}
\mathrm{Ai}_{-1}\left( \nu^{2/3}\zeta \right) 
=\frac{e^{\pi i/3}}{2\pi^{1/2} \nu^{1/6}
\zeta^{1/4}}\exp \left\{ -\nu\xi 
+\sum\limits_{s=1}^{n-1}(-1)^{s}\frac{a_{s}}
{s \nu^{s}\xi^{s}}\right\} 
\\ \times
\left\{ 1+\eta_{n}^{(-1)}(\nu,\xi) \right\},
\end{multline}
and 
\begin{multline}
\label{eq69}
\mathrm{Ai}_{-1}^{\prime }
\left( \nu^{2/3}\zeta \right) 
=-\frac{e^{\pi i/3} \nu^{1/6}\zeta ^{1/4}}{2\pi ^{1/2}}
\exp \left\{ -\nu\xi+\sum\limits_{s=1}^{n-1}
(-1)^{s}\frac{\tilde{a}_{s}}{s \nu^{s}\xi ^{s}}\right\}
\\ \times
\left\{ 1+\tilde{\eta}_{n}^{(-1) }(\nu,\xi) \right\},
\end{multline}
where $\eta_{n}^{(-1)}(\nu,\xi)$ and $\tilde{\eta}_{n}^{(-1) }(\nu,\xi)$ are $\mathcal{O}\{(\nu\xi)^{-n}\}$ as $\nu \xi \to \infty$,  uniformly for a sector containing $\pi  \leq  \arg (\xi)  \leq \frac32\pi$. Explicit and readily computable bounds for the error terms in (\ref{eq63}) - (\ref{eq69}) are furnished by \cite[Thm. 2.4]{Dunster:2021:SEB}.

In order to construct the desired expansions we utilise the well-known connection formula \cite[Eq. 10.27.4]{NIST:DLMF}
\begin{equation}
\label{eq70}
\frac{2 \sinh(\nu \pi)}{\pi i} K_{i\nu}(z)
=I_{i\nu}(z)-I_{-i\nu}(z).
\end{equation}
Also from \cite{Dunster:1990:BFP} we approximate a companion function (see also \cite[Sec. 10.45]{NIST:DLMF}) given by
\begin{equation}
\label{eq71}
\frac{2 \sinh(\nu \pi)}{\pi} L_{i\nu}(z)
=I_{i\nu}(z)+I_{-i\nu}(z).
\end{equation}
We remark that $K_{i\nu}(x)$ and $L_{i\nu}(x)$ are both real-valued and form a numerically satisfactory pair of solutions of the modified Bessel equation for $\nu>0$ and $0<x<\infty$.

We now introduce the following four functions, which will approximate the modified Bessel functions, namely
\begin{equation}
\label{eq59}
w_{l}(\nu,z)
=\mathrm{Ai}_{l} \left(\nu^{2/3}\zeta\right)
A(\nu,z)+\mathrm{Ai}'_{l} 
\left(\nu^{2/3}\zeta \right)B(\nu,z)
\quad (l=0,\pm 1),
\end{equation}
and
\begin{equation}
\label{eq60}
w_{2}(\nu,z)
=\mathrm{Bi} \left(\nu^{2/3}\zeta\right)
A(\nu,z)+\mathrm{Bi}'
\left(\nu^{2/3}\zeta \right)B(\nu,z).
\end{equation}
Here $A(\nu,z)$ and $B(\nu,z)$ are arbitrary, but we now specify them uniquely via the carefully selected pair of equations
\begin{equation}
\label{eq72}
I_{i\nu}(\nu z)
= \frac{2^{1/2} e^{-\pi i/6}  e^{\nu \pi/2}}{\nu^{1/3}}
\left(\frac{\zeta}{z^2-1}\right)^{1/4} w_{1}(\nu,z),
\end{equation}
and
\begin{equation}
\label{eq73}
I_{-i\nu}(\nu z)
=\frac{2^{1/2} e^{\pi i/6}  e^{\nu \pi/2}}{\nu^{1/3}}
\left(\frac{\zeta}{z^2-1}\right)^{1/4} w_{-1}(\nu,z),
\end{equation}
With these choices we shall show that $A(\nu,z)$ and $B(\nu,z)$ are slowly varying for large $\nu$ with simple asymptotic expansions in the right half principal plane. However first we observe from (\ref{eq61}), (\ref{eq70}), (\ref{eq59}), (\ref{eq72}) and (\ref{eq73}) that
\begin{equation}
\label{eq74}
K_{i\nu}(\nu z)
=\frac{\pi e^{\nu \pi/2}}
{2^{1/2}\nu^{1/3}\sinh(\nu \pi)}
\left(\frac{\zeta}{z^2-1}\right)^{1/4} w_{0}(\nu,z).
\end{equation}
Further, from (\ref{eq62}), (\ref{eq71}) and (\ref{eq59}) - (\ref{eq73})
\begin{equation}
\label{eq75}
L_{i\nu}(\nu z)
=\frac{ \pi e^{\nu \pi/2}}
{2^{1/2}\nu^{1/3}\sinh(\nu \pi)}
\left(\frac{\zeta}{z^2-1}\right)^{1/4} w_{2}(\nu,z).
\end{equation}

Next, on solving the pair of equations (\ref{eq72}) and (\ref{eq73}), as well as referring to (\ref{eq59}) and the Airy function Wronskian relation \cite[Eq. 9.2.9]{NIST:DLMF}, we obtain the explicit representations
\begin{multline}
\label{eq76}
A(\nu,z)
=-\frac{2^{1/2} \pi \nu^{1/3}}{e^{\nu \pi/2}}
\left(\frac{z^2-1}{\zeta}\right)^{1/4}
\left\{e^{-\pi i/3}   I_{i\nu}(\nu z)
\mathrm{Ai}_{-1}^{\prime }\left( \nu^{2/3}\zeta\right)
\right.
\\
\left. +e^{\pi i/3} I_{-i\nu}(\nu z)
\mathrm{Ai}_{1}^{\prime }\left( \nu^{2/3}\zeta\right)
\right\},
\end{multline}
and 
\begin{multline}
\label{eq77}
B(\nu,z)
=\frac{2^{1/2} \pi \nu^{1/3}}{e^{\nu \pi/2}}
\left(\frac{z^2-1}{\zeta}\right)^{1/4}
\left\{e^{-\pi i/3}   I_{i\nu}(\nu z)
\mathrm{Ai}_{-1}\left( \nu^{2/3}\zeta\right)
\right.
\\
\left. + e^{\pi i/3} I_{-i\nu}(\nu z)
\mathrm{Ai}_{1}\left( \nu^{2/3}\zeta\right)
\right\}.
\end{multline}
From these it readily verified that both are real for $0<z<\infty$ ($\arg(z)=0$), and analytic in the half plane $|\arg(z)|\leq \frac12 \pi$ except at $z=0$. We shall demonstrate they are also bounded in the principal right half-plane, including at $z=0$.

To this end, from (\ref{eq59}), (\ref{eq72}), (\ref{eq74}) and \cite[Eq. 9.2.8]{NIST:DLMF} we obtain in a similar manner to the derivation of (\ref{eq76}) and (\ref{eq77}) the alternative representations
\begin{multline}
\label{eq78}
A(\nu,z)
=\frac{2^{1/2}\nu^{1/3}}{e^{\nu \pi/2}}
\left(\frac{z^2-1}{\zeta}\right)^{1/4}
\left\{2 e^{-\pi i/6} \sinh(\nu \pi) K_{i\nu}(\nu z)
\mathrm{Ai}_{1}^{\prime }\left( \nu^{2/3}\zeta\right)
\right.
\\
\left. - \pi  I_{i\nu}(\nu z)
\mathrm{Ai}^{\prime }\left( \nu^{2/3}\zeta\right)
\right\},
\end{multline}
and
\begin{multline}
\label{eq79}
B(\nu,z)
=\frac{2^{1/2}\nu^{1/3}}{e^{\nu \pi/2}}
\left(\frac{z^2-1}{\zeta}\right)^{1/4}
\left\{\pi  I_{i\nu}(\nu z)
\mathrm{Ai}\left( \nu^{2/3}\zeta\right)
\right.
\\
\left. 
-2 e^{-\pi i/6} \sinh(\nu \pi) K_{i\nu}(\nu z)
\mathrm{Ai}_{1}\left( \nu^{2/3}\zeta\right)
\right\}.
\end{multline}

The idea is to use (\ref{eq76}) and (\ref{eq77}) in a domain where $I_{i\nu}(\nu z)$ and $I_{-i\nu}(\nu z)$ form a numerically satisfactory pair, namely for $z \in Z_{-1} \cap Z_{1}$, and to use (\ref{eq78}) and (\ref{eq79}) for $z \in Z_{0} \cap Z_{1}$ where $I_{i\nu}(\nu z)$ and $K_{i\nu}(\nu z)$ form a numerically satisfactory pair.

With this in mind, following \cite[Eqs. (1.16) - (1.18)]{Dunster:2021:SEB} for $s=1,2,3,\ldots$ we define
\begin{equation}
\label{eq80}
\mathcal{E}_{s}(z) =\mathrm{E}_{s}(\beta) +
(-1)^{s}a_{s}s^{-1}\xi^{-s},
\end{equation}
and
\begin{equation}
\label{eq81}
\tilde{\mathcal{E}}_{s}(z) =\mathrm{E}_{s}(\beta)
+(-1)^{s}\tilde{a}_{s}s^{-1}\xi^{-s}.
\end{equation}
Let $n=2,3,4,\ldots$; then from (\ref{eq22}) - (\ref{eq24}), (\ref{eq63}), (\ref{eq64}), (\ref{eq66}) - (\ref{eq69}), and (\ref{eq76}) - (\ref{eq81}) we obtain
\begin{multline}
\label{eq82}
A(\nu,z)
=\frac12 \exp \left\{ \sum\limits_{s=1}^{n-1}
\frac{\tilde{\mathcal{E}}_{2s}(z) }{\nu^{2s}}\right\}
\left[  
\exp \left\{ \sum\limits_{s=0}^{n-1}
\frac{\tilde{\mathcal{E}}_{2s+1}(z) }
{\nu^{2s+1}}\right\}
\left(1+\tilde{\epsilon}_{2n,1}(\nu,z)\right)
\right.
\\
\left. 
+\exp \left\{ -\sum\limits_{s=0}^{n-1}
\frac{\tilde{\mathcal{E}}_{2s+1}(z) }
{\nu^{2s+1}}\right\}
\left(1+\tilde{\epsilon}_{2n,2}(\nu,z)\right)
 \right],
\end{multline}
and
\begin{multline}
\label{eq83}
B(\nu,z)
=\frac{1}{2\nu^{1/3}\zeta^{1/2}} 
\exp \left\{ \sum\limits_{s=1}^{n-1}
\frac{\mathcal{E}_{2s}(z) }{\nu^{2s}}\right\}
\left[ 
\exp \left\{\sum\limits_{s=0}^{n-1}
\frac{\mathcal{E}_{2s+1}(z) }
{\nu^{2s+1}}\right\}
\left(1+\epsilon_{2n,1}(\nu,z)\right)
\right.
\\
\left. 
-\exp \left\{-\sum\limits_{s=0}^{n-1}
\frac{\mathcal{E}_{2s+1}(z) }
{\nu^{2s+1}}\right\}
\left(1+\epsilon_{2n,2}(\nu,z)\right)
\right],
\end{multline}
where for $z \in Z_{0} \cap Z_{1}$
\begin{equation}
\label{eq84}
\epsilon_{n,1}(\nu,z)
=\frac{ \{1+\eta_{n,1}(\nu,z)\}
\{1+\eta_{n}^{(0)}(\nu,\xi)\}}
{1+\eta_{n,1}(\nu,\infty)}-1,
\end{equation}
\begin{equation}
\label{eq85}
\epsilon_{n,2}(\nu,z)
=\{1-e^{-2\nu \pi}\}
\{ 1+\eta_{n,0}(\nu,z)\}
\{ 1+\eta_{n}^{(1) }(\nu,\xi)\}-1,
\end{equation}
\begin{equation}
\label{eq86}
\tilde{\epsilon}_{n,1}(\nu,z)
=\frac{ \{1+\eta_{n,1}(\nu,z)\}
\{1+\tilde{\eta}_{n}^{(0)}(\nu,\xi)\}}
{1+\eta_{n,1}(\nu,\infty)}-1,
\end{equation}
\begin{equation}
\label{eq87}
\tilde{\epsilon}_{n,2}(\nu,z)
=\{1-e^{-2\nu \pi}\}
\{ 1+\eta_{n,0}(\nu,z)\}
\{ 1+\tilde{\eta}_{n}^{(1) }(\nu,\xi)\}-1,
\end{equation}
and for $z \in Z_{-1} \cap Z_{1}$
\begin{equation}
\label{eq88}
\epsilon_{n,1}(\nu,z)
=\frac{ \{1+\eta_{n,1}(\nu,z)\}
\{1+\eta_{n}^{(-1)}(\nu,\xi)\}}
{1+\eta_{n,1}(\nu,\infty)}-1,
\end{equation}
\begin{equation}
\label{eq89}
\epsilon_{n,2}(\nu,z)
=\frac{\{ 1+\eta_{n,-1}(\nu,z)\}
\{ 1+\eta_{n}^{(1)}(\nu,\xi)\}}
{1+\eta_{n,-1}(\nu,\infty e^{\pi i})}-1,
\end{equation}
\begin{equation}
\label{eq90}
\tilde{\epsilon}_{n,1}(\nu,z)
=\frac{ \{1+\eta_{n,1}(\nu,z)\}
\{1+\tilde{\eta}_{n}^{(-1)}(\nu,\xi)\}}
{1+\eta_{n,1}(\nu,\infty)}-1,
\end{equation}
and
\begin{equation}
\label{eq91}
\tilde{\epsilon}_{n,2}(\nu,z)
=\frac{\{ 1+\eta_{n,-1}(\nu,z)\}
\{ 1+\tilde{\eta}_{n}^{(1) }(\nu,\xi)\}}
{1+\eta_{n,-1}(\nu,\infty e^{\pi i})}-1.
\end{equation}

This leads to the following main result.

\begin{theorem}
\label{thm:AB-expansions}
The modified Bessel functions $I_{\pm i\nu}(\nu z)$, $K_{i\nu}(\nu z)$ and $L_{i\nu}(\nu z)$ are expressible in the forms (\ref{eq72}) - (\ref{eq75}), where $w_{j}(\nu,z)$ ($j=0,\pm 1, 2$) are defined by (\ref{eq59}) and (\ref{eq60}). In these $A(\nu,z)$ and $B(\nu,z)$ are analytic for all nonzero $z$ in the half-plane $|\arg(z)|\leq \frac{1}{2}\pi$, and possess the asymptotic expansions
\begin{equation}
\label{eq92}
A(\nu,z) \sim 
\exp \left\{ \sum\limits_{s=1}^{\infty}
\frac{\tilde{\mathcal{E}}_{2s}(z) }{\nu^{2s}}\right\} 
\cosh \left\{ \sum\limits_{s=0}^{\infty}
\frac{\tilde{\mathcal{E}}_{2s+1}(z) }
{\nu^{2s+1}}\right\},
\end{equation}
and
\begin{equation}
\label{eq93}
B(\nu,z) \sim 
\frac{1}{\nu^{1/3} \zeta^{1/2}}
\exp \left\{ \sum\limits_{s=1}^{\infty}
\frac{\mathcal{E}_{2s}(z) }{\nu^{2s}}\right\} 
\sinh \left\{ \sum\limits_{s=0}^{\infty}
\frac{\mathcal{E}_{2s+1}(z) }
{\nu^{2s+1}}\right\},
\end{equation}
as $\nu \to \infty$, uniformly for $|\arg(z)|\leq \frac{1}{2}\pi$ with an arbitrarily small compact neighbourhood of $z= 1$ being excluded. Here $\zeta$ is defined by (\ref{eq01}), and the coefficients $\mathcal{E}_{s}(z)$ are defined by (\ref{eq09}) - (\ref{eq11}), (\ref{eq65}), (\ref{eq80}) and (\ref{eq81}).
\end{theorem}

\begin{proof}
The region $\{Z_{0} \cap Z_{1}\}\cup \{Z_{-} \cap Z_{1}\}$ includes the first quadrant $0 \leq \arg(z) \leq \frac12 \pi$ except the turning point $z=1$, and thus from (\ref{eq84}) - (\ref{eq91}) the  $\epsilon$ and $\tilde{\epsilon}$ error terms in (\ref{eq82}) and (\ref{eq83}) are $\mathcal{O}(\nu^{-n})$ as $\nu \to \infty$ in this unbounded region. The expansions (\ref{eq92}) and (\ref{eq93}) then follow from (\ref{eq82}) and (\ref{eq83}) for $0 \leq \arg(z) \leq \frac{1}{2}\pi$ ($z \neq 1$). Since both coefficient functions $A(\nu,z)$ and $B(\nu,z)$ are analytic and real on the positive real $z$ axis, as are the coefficients in the four asymptotic sums therein, Schwarz symmetry extends these expansions to the half-plane  $|\arg(z)|\leq \frac12 \pi$, ($z \neq 1$).
\end{proof}

Being analytic, $A(\nu,z)$ and $B(\nu,z)$ are bounded at $z=1$ even though the expansions break down there. Asymptotic approximation of  $A(\nu,z)$ and $B(\nu,z)$ near, or at, $z=1$ can be achieved in two ways, either using the above expansions in conjunction with Cauchy's integral formula, as described in \cite[Thm. 4.2]{Dunster:2021:SEB}, or by expanding them as regular expansions in inverse powers of $\nu$, which we describe next.

Now, similarly to \cite[Eqs. (4.9) and (4.20)]{Dunster:2024:AEG} let $\tilde{d}_{1}(z)=\tilde{\mathcal{E}}_{1}(z)$ and
\begin{equation}
\label{eq94}
\tilde{d}_{s}(z)=\tilde{\mathcal{E}}_{s}(z)
+\frac{1}{s}\sum_{j=1}^{s-1}j
\tilde{\mathcal{E}}_{j}(z)\tilde{d}_{s-j}(z)
\quad (s=2,3,4,\ldots).
\end{equation}
Then for $\mathrm{A}_{s}(z)=\tilde{d}_{2s}(z)$ ($s=1,2,3,\ldots$) one can show from (\ref{eq92}) that
\begin{equation}
\label{eq95}
A(\nu,z) \sim 
1+\sum\limits_{s=1}^{\infty}
\frac{\mathrm{A}_{s}(z) }{\nu^{2s}}
\quad (\nu \to \infty).
\end{equation}

Similarly, let $d_{1}(z)=\mathcal{E}_{1}(z)$ and
\begin{equation}
\label{eq96}
d_{s}(z)=\mathcal{E}_{s}(z)
+\frac{1}{s}\sum_{j=1}^{s-1}j
\mathcal{E}_{j}(z)d_{s-j}(z)
\quad (s=2,3,4,\ldots).
\end{equation}
Then for $\mathrm{B}_{s}(z)=\zeta^{-1/2}d_{2s+1}(z)$ ($s=0,1,2,\ldots$)
\begin{equation}
\label{eq97}
B(\nu,z) \sim \frac{1}{\nu^{4/3}}
\sum\limits_{s=0}^{\infty}
\frac{\mathrm{B}_{s}(z)}{\nu^{2s}}
\quad (\nu \to \infty).
\end{equation}

From \cite[Thm. 3.1]{Dunster:2021:NKF} the coefficients $\mathrm{A}_{s}(z)$ and $\mathrm{B}_{s}(z)$ have a removable singularity at $z=1$, and both expansions are valid in a neighbourhood of $z=1$ (in fact throughout the half-plane $|\arg(z)|\leq \frac12 \pi$). The coefficients can be readily computed either directly if not too close to the turning point, or by a Taylor series otherwise.

The first few coefficients of both are found to be given by
\begin{equation}
\label{eq98}
\mathrm{A}_{1}(z) 
=\frac12\left\{
\tilde{\mathcal{E}}_{1}^{2}(z)
+2\tilde{\mathcal{E}}_{2}(z)
\right\},
\end{equation}
\begin{equation}
\label{eq99}
\mathrm{A}_{2}(z) 
=\frac{1}{24}\left\{
\tilde{\mathcal{E}}_{1}^{4}(z)
+12\tilde{\mathcal{E}}_{1}^{2}(z)
\tilde{\mathcal{E}}_{2}(z)
+24\tilde{\mathcal{E}}_{1}(z)\tilde{\mathcal{E}}_{3}(z)
+12\tilde{\mathcal{E}}_{2}^{2}(z)
+24\tilde{\mathcal{E}}_{4}(z)
\right\},
\end{equation}
\begin{multline}  
\label{eq100}
\mathrm{A}_{3}(z) 
=\frac{1}{720}\left\{
\tilde{\mathcal{E}}_{1}^6(z)
+ 30 \tilde{\mathcal{E}}_{1}^4(z) \tilde{\mathcal{E}}_{2}(z)
+ 120 \tilde{\mathcal{E}}_{1}^3(z) \tilde{\mathcal{E}}_{3}(z)
+ 180\left(\tilde{\mathcal{E}}_{2}^2(z)
+ 2 \tilde{\mathcal{E}}_{4}(z)\right) \right.
\\
\tilde{\mathcal{E}}_{1}^2(z) 
+ 720\tilde{\mathcal{E}}_{1}(z)\left(\tilde{\mathcal{E}}_{2}(z)
\tilde{\mathcal{E}}_{3}(z)
+ \tilde{\mathcal{E}}_{5}(z)\right)
+ 120 \tilde{\mathcal{E}}_{2}^3(z)
\\
\left.
+ 720\tilde{\mathcal{E}}_{2}\tilde{\mathcal{E}}_{4}(z)
+ 360 \tilde{\mathcal{E}}_{3}^2(z)
+ 720 \tilde{\mathcal{E}}_{6}(z)
\right\},
\end{multline}
\begin{equation}
\label{eq101}
\mathrm{B}_{0}(z) 
=\frac{\mathcal{E}_{1}(z)}{\zeta^{1/2}},
\end{equation}
\begin{equation}
\label{eq102}
\mathrm{B}_{1}(z) 
=\frac{1}{6\zeta^{1/2}}\left\{
\mathcal{E}_{1}^{3}(z)
+6 \mathcal{E}_{1}(z) \mathcal{E}_{2}(z)
+6 \mathcal{E}_{3}(z)
\right\},
\end{equation}
\begin{multline}  
\label{eq103a}
\mathrm{B}_{2}(z) 
=\frac{1}{120 \zeta^{1/2}}\left\{
\mathcal{E}_{1}^{5}(z)
+ 20 \mathcal{E}_{1}^{3}(z) \mathcal{E}_{2}(z) 
+ 60 \mathcal{E}_{1}^{2}(z) \mathcal{E}_{3}(z) 
 \right.
\\
\left.
+ 60  \mathcal{E}_{1}(z) \left(\mathcal{E}_{2}^{2}(z)
+ 2 \mathcal{E}_{4}(z) \right)
+ 120 \mathcal{E}_{2}(z) \mathcal{E}_{3}(z) 
+ 120 \mathcal{E}_{5}(z) 
\right\},
\end{multline}
and
\begin{multline}  
\label{eq103}
\mathrm{B}_{3}(z) 
=\frac{1}{5040 \zeta^{1/2}}\left\{
\mathcal{E}_{1}^{7}(z)
+ 42 \mathcal{E}_{1}^{5}(z) \mathcal{E}_{2}(z) 
+210 \mathcal{E}_{1}^{4}(z) \mathcal{E}_{3}(z) 
 \right.
\\
+420\mathcal{E}_{1}^{3}(z) 
\left(\mathcal{E}_{2}^{2}(z)+2\mathcal{E}_{4}(z)\right)
+2520\mathcal{E}_{1}^{2}(z)\left(\mathcal{E}_{2}(z)
\mathcal{E}_{3}(z) 
+ \mathcal{E}_{5}(z)\right) 
\\
+ 840 \mathcal{E}_{1}(z)\left(\mathcal{E}_{2}^{3}(z) 
+ 6\mathcal{E}_{2}(z)\mathcal{E}_{4}(z) 
+ 3\mathcal{E}_{3}^{2}(z) + 6\mathcal{E}_{6}(z)\right)
\\
\left.
+ 2520\mathcal{E}_{2}^{2}(z)\mathcal{E}_{3}(z) 
+ 5040\mathcal{E}_{2}(z)\mathcal{E}_{5}(z) 
+ 5040\mathcal{E}_{3}(z)\mathcal{E}_{4}(z) 
+ 5040\mathcal{E}_{7}(z)
\right\}.
\end{multline}

As $z \to 1$ we note that
\begin{multline}
\label{eq104}
\mathrm{A}_{1}(z) 
=\tfrac{1}{225}
+\mathcal{O}(z-1),
\, \mathrm{A}_{2}(z) 
=\tfrac{151439}{218295000}
+\mathcal{O}(z-1),
\\ \mathrm{A}_{3}(z) 
=\tfrac{1374085664813273149}{3633280647121125000000}
+\mathcal{O}(z-1),
\end{multline}
and
\begin{multline}
\label{eq107}
\mathrm{B}_{0}(z) 
=\tfrac{1}{70} \,2^{1/3}
+\mathcal{O}(z-1),
\, \mathrm{B}_{1}(z) 
=\tfrac{1213}{1023750} \,2^{1/3}
+\mathcal{O}(z-1),
\\
\mathrm{B}_{2}(z) 
=\tfrac{9597171184603}
{25476663712500000} \,2^{1/3}
+\mathcal{O}(z-1),
\\
\mathrm{B}_{3}(z) 
=\tfrac{53299328587804322691259}
{91182706744837207500000000} \,2^{1/3}
+\mathcal{O}(z-1),
\end{multline}
in accord with them having a removable singularity at $z=1$.

Let us numerically check the accuracy of (\ref{eq74}) for $z=x>0$. In this we use the expansions (\ref{eq95}) and (\ref{eq97}) and compute the coefficients via Taylor series for $|x-1|<0.1$. For $x\geq 1.1$ we compute the coefficients directly from (\ref{eq98}) - (\ref{eq103}). For $0<x\leq 0.9$ these expressions remain valid, but care must be taken since $\beta$ and $\xi$ are imaginary (whereas the coefficients $\mathrm{A}_{s}(x)$ and $\mathrm{B}_{s}(x)$ are real). In this instance we find it convenient to use (\ref{eq01}) and (\ref{eq06}) to express $\xi$ in the form
\begin{equation}
\label{eq108}
\xi=\frac{1-\tau}{\beta},
\end{equation}
where
\begin{equation}
\label{eq109}
\tau=\tau(x)=\frac{\ln\left\{\left(1-x^2\right)^{1/2}
+1\right\}-\ln(x)}{\left(1-x^2\right)^{1/2}}>0
\quad (0<x< 1).
\end{equation}

Furthermore, from (\ref{eq01}) and (\ref{eq109}),
\begin{equation}
\label{eq110}
\frac{1}{\zeta^{1/2}}=\frac{2\zeta}{3\xi}
=-\frac{2\beta \zeta}{3(\tau-1)},
\end{equation}
which we insert into our expressions for $\mathrm{B}_{s}(x)$. As a result, both sets of coefficients involve the real-valued functions $\tau$ and $\zeta$, along with only even powers of $\beta$, and since $\beta^2=(x^2-1)^{-1}$ is negative these too are real. For example, we find that
\begin{equation}
\label{eq111a}
\mathrm{A}_{1}(x)
=-\frac{385}{1152\left(1-x^{2}\right)^{3}}
+\frac{7(99\tau-104)}{1728\left(1-x^{2}\right)^{2}(\tau-1)}
-\frac{729\tau^2 - 1584\tau + 400}
{10368\left(1-x^{2}\right)(\tau-1)^{2}},
\end{equation}
and
\begin{equation}
\label{eq111b}
\mathrm{B}_{0}(x)
=-\frac{5\zeta}{36\left(1-x^{2}\right)^{2}
(\tau-1)}+\frac{(9\tau-4)\zeta}
{108\left(1-x^{2}\right)(\tau-1)^{2}},
\end{equation}
which we use (and similarly for the subsequent coefficients) when $0<x \leq 0.9$.

Typically, the relative error of an approximation $\hat{f}$ of a non-vanishing function $f$ is made by computing $|(f-\hat{f})/f
|$. Naturally, this is not feasible in regions where the function has zeros. Thus, in order to accurately estimate the relative error, we define a so-called envelope function for $K_{i\nu}(\nu x)$ by
\begin{equation}
\label{eq111c}
N(\nu,x)
=
    \begin{cases}
        \left\{K_{i\nu}^{2}(\nu x)+L_{i\nu}^{2}(\nu x)\right\}^{1/2} & (0<x \leq l_{\nu,1})\\
        K_{i\nu}(\nu x) & (l_{\nu,1} < x < \infty)
    \end{cases},
\end{equation}
where $l_{\nu,1}$ is the largest positive zero of $L_{i\nu}(\nu x)$. In the oscillatory interval $(0,l_{\nu,1})$ this positive function is close to the amplitude of $K_{i\nu}(\nu x)$, whilst of course being identically equal to the function in the non-oscillatory interval. For example, in \cref{fig:envK} the function $K_{i \nu}(\nu x)$ (solid curve) and its envelope $N(\nu,x)$ (dashed curve) are shown for $\nu=10$.

Now, if we take four terms in each of the series (\ref{eq95}) and (\ref{eq97}), and refer to (\ref{eq59}) and (\ref{eq74}), an appropriately scaled error function for that approximation is given by
\begin{equation}
\label{eq111d}
\chi(\nu,x)=\frac{\left|K_{i\nu}(\nu x)-S(\nu,x)\right|}
{N(\nu,x)},
\end{equation}
where
\begin{multline}
\label{eq111e}
S(\nu,x)
=\frac{ \pi e^{\nu \pi/2}}
{2^{1/2}\nu^{1/3}\sinh(\nu \pi)}
\left(\frac{\zeta}{x^2-1}\right)^{1/4} 
\left[\mathrm{Ai}\left(\nu^{2/3}\zeta\right)
\left\{1+\sum\limits_{s=1}^{3}
\frac{\mathrm{A}_{s}(x) }{\nu^{2s}}\right\}
\right.
\\ \left.
+\frac{\mathrm{Ai}'\left(\nu^{2/3}\zeta \right)}
{\nu^{4/3}}\sum\limits_{s=0}^{3}
\frac{\mathrm{B}_{s}(x)}{\nu^{2s}} \right].
\end{multline}

A graph of 
\begin{equation}
\label{eq111f}
\Omega_{3}(\nu,x) = \log_{10}\left\{\chi(\nu,x)\right\}
\end{equation}
is given in \cref{fig:AiryK} for $\nu=10$, and we observe good accuracy uniformly for all positive $x$.

\begin{figure}
 \centering
 \includegraphics[
 width=0.7\textwidth,keepaspectratio]{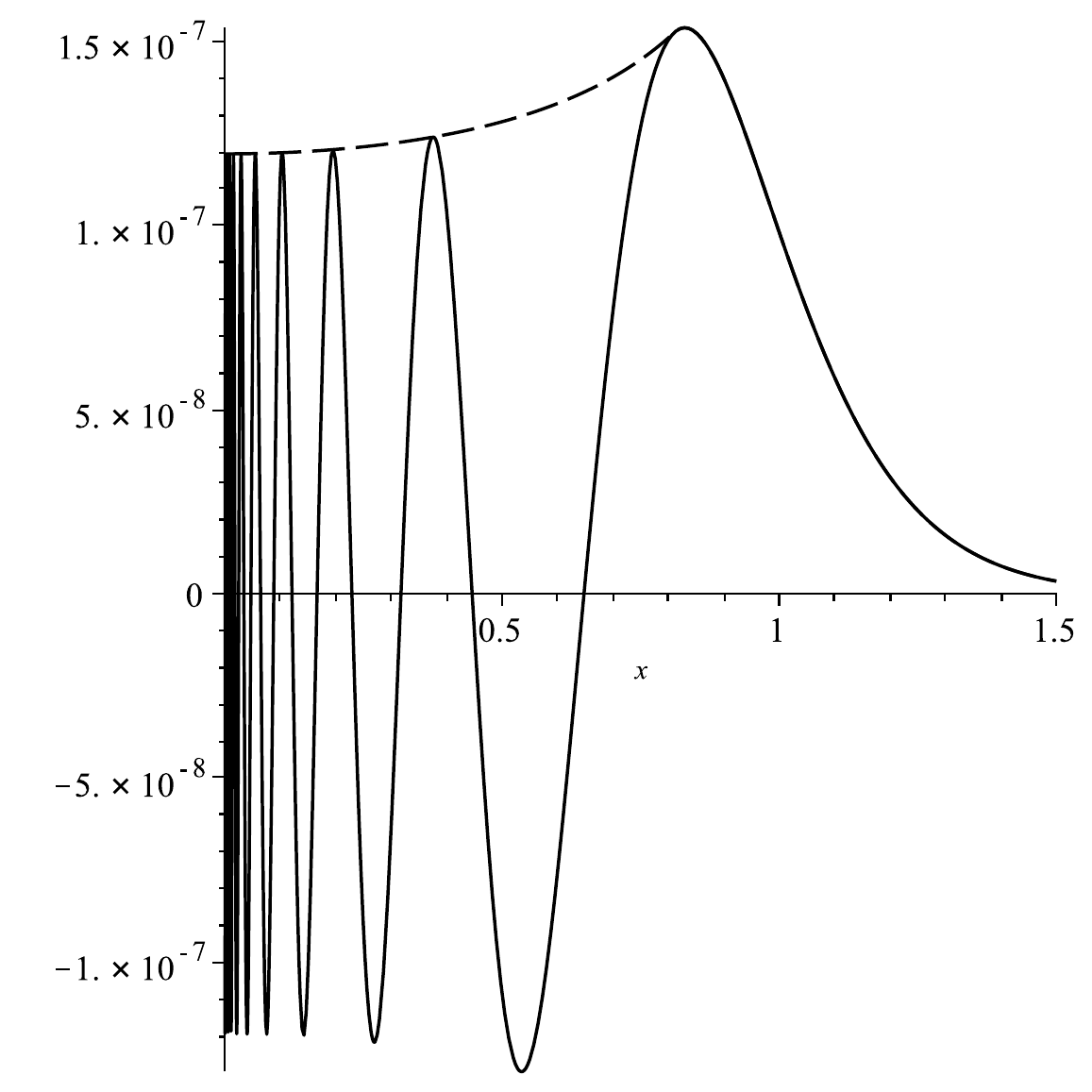}
 \caption{Graphs of $K_{10i}(10x)$ (solid curve) and $N(10,x)$ (dashed curve).}
 \label{fig:envK}
\end{figure}

\begin{figure}
 \centering
 \includegraphics[
 width=0.7\textwidth,keepaspectratio]{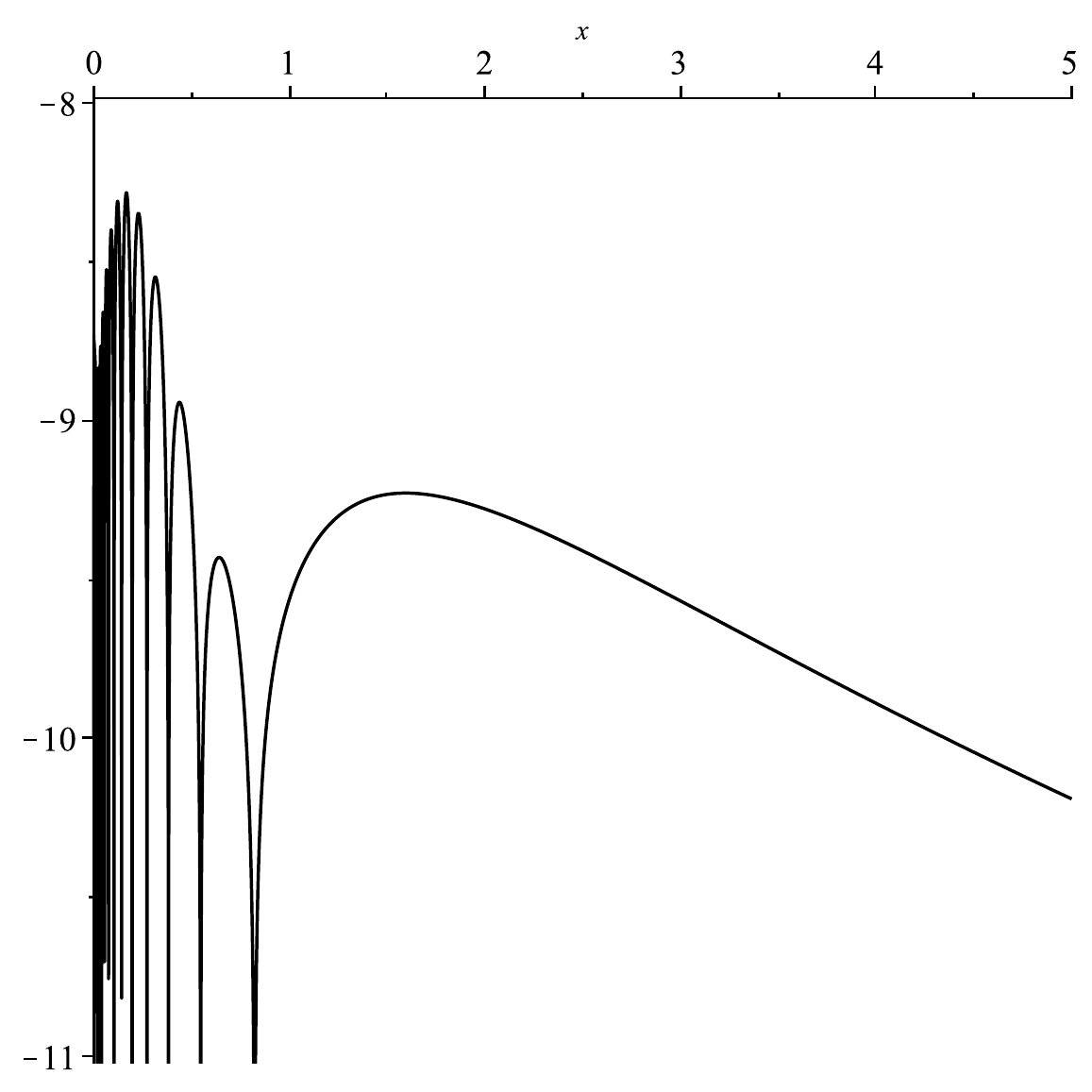}
 \caption{Graph of $\Omega_{3}(10,x)$.}
 \label{fig:AiryK}
\end{figure}

\section{Positive zeros of \texorpdfstring{$K_{i \nu}(t)$}{} and related modified Bessel functions}
\label{sec:Kzeros}

We use the method of \cite{Dunster:2024:AZB} where uniform asymptotic expansions were derived for the zeros of (non modified) Bessel functions of large positive order. Although the expansions we derive here are very similar and equally powerful, the fundamental difference is that the positive zeros of modified Bessel functions of imaginary order are bounded and approach zero in a limit, as opposed to Bessel functions of real order where they are unbounded.

Now, similarly to \cite[Eq. (3.9)]{Dunster:2024:AZB}, for $\nu >0$ define functions $\mathcal{Y}(\nu,z)$ and $\mathcal{Z}(\nu,z)$ via the pair of equations
\begin{equation}
\label{eq112}
I_{\pm i\nu}(\nu z)
= e^{\mp \pi i/6}
\mathcal{Y}(\nu,z)
\mathrm{Ai}_{\pm 1}\left( \nu^{2/3}
\mathcal{Z}(\nu,z)\right).
\end{equation}
Compare these to (\ref{eq59}), (\ref{eq72}) and (\ref{eq73}). From the Wronskians \cite[Eqs. 9.2.9 and 10.28.1]{NIST:DLMF} one can show that
\begin{equation}
\label{eq113}
\mathcal{Y}(u,z)
=\frac{2}{\nu^{1/3}}
\left\{\frac{\sinh(\nu \pi)}
{z \,\partial \mathcal{Z}(u,z)/\partial z}\right\}^{1/2},
\end{equation}
although we do not require this function.

The reason we use the forms (\ref{eq112}) is that with suitable information on $\mathcal{Z}(\nu,z)$ they are most practicable for constructing asymptotic expansions for the zeros of the Bessel functions in question. In short, the Bessel zeros occur \emph{exactly} when $\nu^{2/3}\mathcal{Z}(\nu,z)$ equals a zero of the corresponding Airy function in the expression, and these latter zeros are well known and readily computable (primarily because Airy functions only depend on one variable, as opposed to Bessel functions which have an addition parameter).

Now from (\ref{eq61}), (\ref{eq70}) and (\ref{eq112}) we have
\begin{equation}
\label{eq114}
\frac{2 \sinh(\nu \pi)}{\pi} K_{i\nu}(\nu z)
= \mathcal{Y}(\nu,z)
\mathrm{Ai}\left( \nu^{2/3}
\mathcal{Z}(\nu,z)\right),
\end{equation}
and this is the relation that we primarily focus on (asymptotic expansions for zeros of the other Bessel functions can similarly be obtained, and we will also consider those of $L_{i\nu}(\nu z)$).

We now apply \cite[Thm. 2.2]{Dunster:2024:AZB} to obtain
\begin{equation}
\label{eq115}
\mathcal{Z}(\nu,z) \sim \zeta
+\sum_{s=1}^{\infty}
\frac{\Upsilon_{s}(z)}{\nu^{2s}}
\quad (\nu \to \infty),
\end{equation}
where $\Upsilon_{s}(z)$ are analytic at the turning point $z=1$. This expansion is uniformly valid in the same region as the expansions (\ref{eq95}) and (\ref{eq97}), which includes the half-plane $|\arg(z)|\leq \frac{1}{2}\pi$. In particular, it is uniformly valid for $z \in (0,1)$ where the zeros are located.

The coefficients $\Upsilon_{s}(z)$ follow from \cite[Thm. 2.2]{Dunster:2024:AZB}, with the leading one given by Eq. (2.28) of that reference, namely
\begin{equation}
\label{eq116}
\Upsilon_{1}(z)=
\frac{3 \xi \mathrm{E}_{1}(\beta)}
{2\zeta^{2}}
-\frac{5}{48\zeta^{2}},
\end{equation}
where in the present notation we have replaced $\hat{E}_{1}(z)$ by $\mathrm{E}_{1}(\beta)$ (see (\ref{eq09})).

The subsequent coefficients are given recursively as described in the referenced theorem. The next three are explicitly given by \cite[Eqs. (2.40) - (2.42)]{Dunster:2024:AZB}, again with $\hat{E}_{s}(z)$ replaced by $\mathrm{E}_{s}(\beta)$ ($s=3,5,7$). Here they differ in that $\xi$, $\zeta$ and $\mathrm{E}_{s}(\beta)$ are not the same as in that paper, although the present ones are closely related to the corresponding functions in that reference.

At this stage it is convenient to work with the variable defined by
\begin{equation}
\label{eq117}
\sigma(z)=\zeta^{1/2}\beta
=\left(\frac{\zeta}{z^{2}-1}\right)^{1/2},
\end{equation}
which is analytic at $z=1$; indeed from (\ref{eq01a}) one obtains
\begin{equation}
\label{eq118}
\sigma(z)
=\frac{1}{2^{1/3}}
+\frac{{2}^{2/3}}{5}(z-1)
+\mathcal{O}\left\{(z-1)^2\right\}
\quad (z \to 1).
\end{equation}
Now from (\ref{eq09}), (\ref{eq116}) and (\ref{eq117})
\begin{equation}
\label{eq119}
\Upsilon_{1}(z)
=\frac{1}{ 48\zeta^{2}}
\left\{10\sigma^{3}+6\sigma\zeta-5
\right\},
\end{equation}
and from this we find from (\ref{eq01a})
\begin{equation}
\label{eq120}
\Upsilon_{1}(z)=
\tfrac{1}{70} 2^{1/3}
+\mathcal{O}(z-1)
\quad (z \to 1),
\end{equation}
illustrating that it (like all the other coefficients) is analytic at $z=1$.

Next, from (\ref{eq114}) we see that the positive zeros of $K_{i\nu}(t)$, which we denote by $k_{\nu,m}$ ($m=1,2,3,\ldots$), satisfy the implicit equation
\begin{equation}
\label{eq121}
\mathcal{Z}(\nu,\nu^{-1}k_{\nu,m})
=\nu^{-2/3}\mathrm{a}_{m}
\quad (m=1,2,3,\ldots),
\end{equation}
where $x=\mathrm{a}_{m}$ is the $m$th negative zero of $\mathrm{Ai}(x)$. Note that $k_{\nu,1}>k_{\nu,2}>k_{\nu,3}> \ldots$ and $k_{\nu,m} \to 0$ as $m \to \infty$.

Following the same steps leading to \cite[Thm. 3.1]{Dunster:2024:AZB} we obtain a uniform asymptotic expansion of the form
\begin{equation}
\label{eq122}
k_{\nu,m} \sim
\nu \sum_{s=0}^{\infty}
\frac{\kappa_{m,s}}{\nu^{2s}}
\quad (\nu \to \infty, \, m=1,2,3,\ldots).
\end{equation}
Plugging this into (\ref{eq121}), using (\ref{eq03}) and (\ref{eq115}), re-expanding in inverse powers of $\nu$, and then equating like powers, we deduce for each $m$ the leading term $\kappa_{m,0} \in (0,1)$ is given implicitly by
\begin{equation}
\label{eq123}
\ln\left\{
\frac{1+(1-\kappa_{m,0}^{2})^{1/2}}{\kappa_{m,0}}
\right\}-\left(1-\kappa_{m,0}^{2}\right)^{1/2}
=\frac{2}{3\nu}|\mathrm{a}_{m}|^{3/2}.
\end{equation}
Compare this to the leading term $z_{m,0}$ given by \cite[Eq. (3.31)]{Dunster:2024:AZB}, which unlike $\kappa_{m,0}$ is unbounded and lies in $(1,\infty)$.

From \cite[Eq. 9.9.6]{NIST:DLMF} $|\mathrm{a}_{m}|^{3/2} \sim \frac32 m\pi$ as $m \to \infty$. Thus from (\ref{eq123}) we see that $\kappa_{m,0} \to 0$ as $m/\nu \to \infty$ such that
\begin{equation}
\label{eq124}
\kappa_{m,0}=\frac{2}{e U}
+\mathcal{O}\left(\frac{1}{U^3}\right),
\end{equation}
where
\begin{equation}
\label{eq125}
U=\exp\left\{
\frac{2}{3\nu}|\mathrm{a}_{m}|^{3/2}\right\} \to \infty.
\end{equation}
Also observe that $\kappa_{m,0} \to 1$ as $m/\nu \to 0$.

The subsequent coefficients in (\ref{eq122}) are precisely those given by \cite[Thm. 3.1]{Dunster:2024:AZB}, except that one replaces $z_{m,s}$ by $\kappa_{m,s}$ (and of course our variables $\zeta$ and $\Upsilon_{s}(z)$ differ). In these expressions we compute all $z$ derivatives via (see (\ref{eq03}) and (\ref{eq117}))
\begin{equation}
\label{eq126}
\zeta'=\frac{1}{z\sigma},
\; \sigma'=\frac{1-2z^{2} \sigma^{3}}
{ 2z \zeta},
\end{equation}
with higher derivatives obtained by repeated use of the chain rule.

The coefficients are then found to be given explicitly of the form $\kappa_{m,s}=\hat{\kappa}_{s}(\kappa_{m,0})$, where $\hat{\kappa}_{s}(z)$ are polynomials in $z$, $\sigma(z)$ and $1/\zeta(z)$, each having a removable singularity at the turning point $z=1$ ($\zeta=0$). The leading term of course is $\hat{\kappa}_{0}(z)=z$, with the next two\footnote{These two, along with $\hat{\kappa}_{3}(z)$ and $\hat{\kappa}_{4}(z)$, are available in Maple input format at the GitHub link \url{https://github.com/tmdunster/Bessel-Imaginary-Order}} being given by
\begin{equation}
\label{eq128}
\hat{\kappa}_{1}(z)=
\frac{z \sigma}{48  \zeta^{2}}
\left\{ 5- 10\sigma^3- 6\sigma \zeta \right\},
\end{equation}
and
\begin{multline}
\label{eq128a}
\hat{\kappa}_{2}(z)=
-\frac{z \sigma}{46080
\zeta^{5}}
\left\{200\sigma^{9}\left( 35z^{2}
+221 \right) 
+80\sigma^{7}  \zeta\left( 75z^{2}+982\right) 
\right.
\\
-4000\sigma^{6}  z^{2}
+24\sigma^{5}  \zeta^{2}\left( 45 z^{2}+1543\right)
-200\sigma^{4}  \zeta
\left(6 z^{2}-5 \right) 
 \\
 \left.
+600\sigma^{2}  \zeta^{2}
+10 \sigma^{3}\left(25 z^{2} 
+264 \zeta^{3}\right) 
-250  \sigma \zeta-5525
\right\}.
\end{multline}
Compare these to \cite[Eqs. (3.56) and (3.57)]{Dunster:2024:AZB}.

As $z \to 1$ we find from (\ref{eq03}), (\ref{eq117}), (\ref{eq128}) and (\ref{eq128a}) that $\hat{\kappa}_{1}(z)=-\frac{1}{70}+\mathcal{O}(z-1)$ and $\hat{\kappa}_{2}(z)=-\frac{3781}{3185000}+\mathcal{O}(z-1)$. As $z \to 0^{+}$ it can also be shown that $\hat{\kappa}_{1}(z)=-\frac{1}{12}z+\mathcal{O}\{z/\ln(z)\}$ and $\hat{\kappa}_{2}(z)=\frac{1}{1440}z+\mathcal{O}\{z/\ln(z)\}$.

The graphs of $\hat{\kappa}_{s}(z)/z$ for $s=1,2,3,4$ are shown in \cref{fig:kappa1,fig:kappa2-4} for $s=1$ and $s=2,3,4$ respectively, illustrating their relative magnitudes to the leading term $\hat{\kappa}_{0}(z)=z$.

\begin{figure}
 \centering
 \includegraphics[
 width=0.7\textwidth,keepaspectratio]{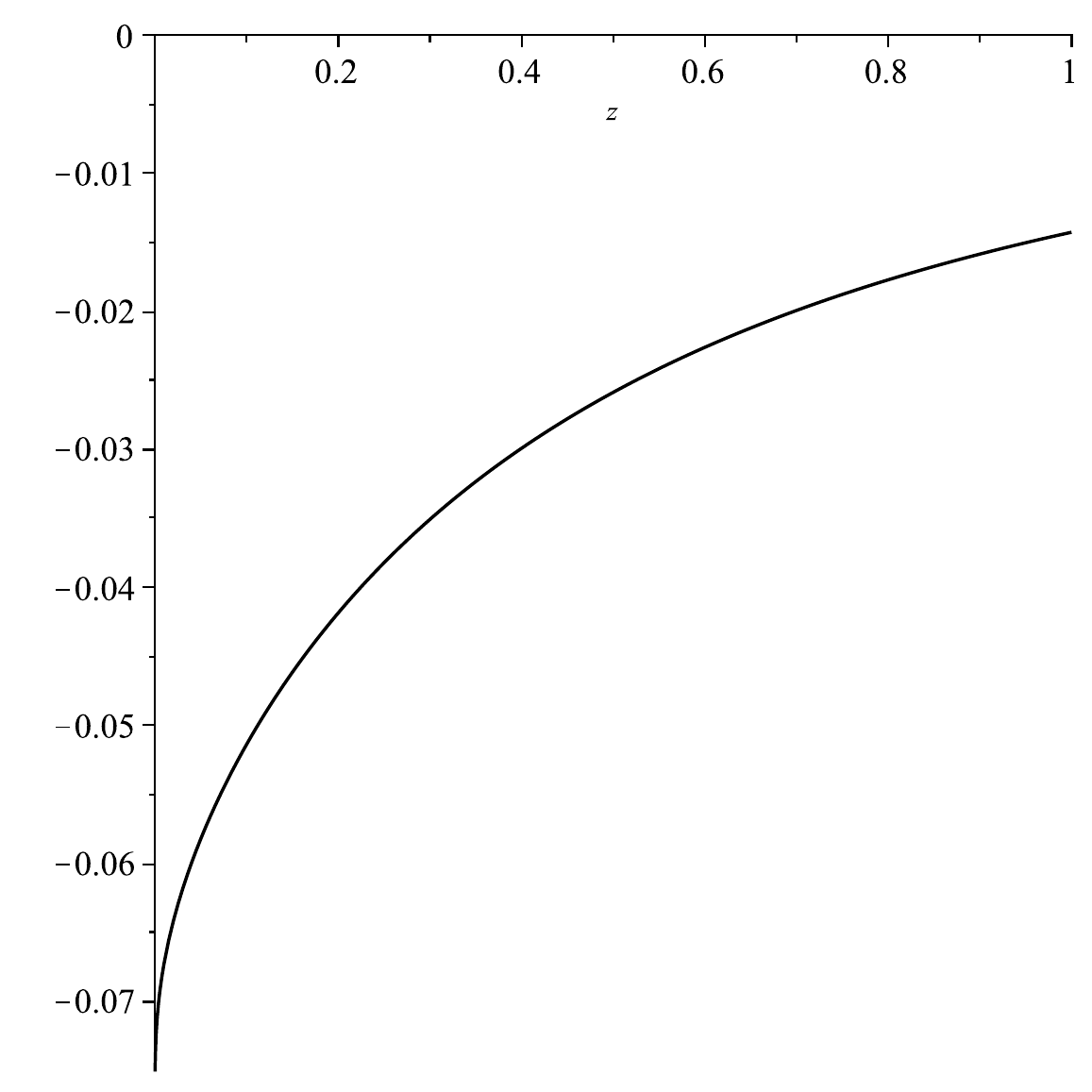}
 \caption{Graph of $\hat{\kappa}_{1}(z)/z$ for $0 \leq z \leq 1$.}
 \label{fig:kappa1}
\end{figure}

\begin{figure}
 \centering
 \includegraphics[
 width=0.7\textwidth,keepaspectratio]{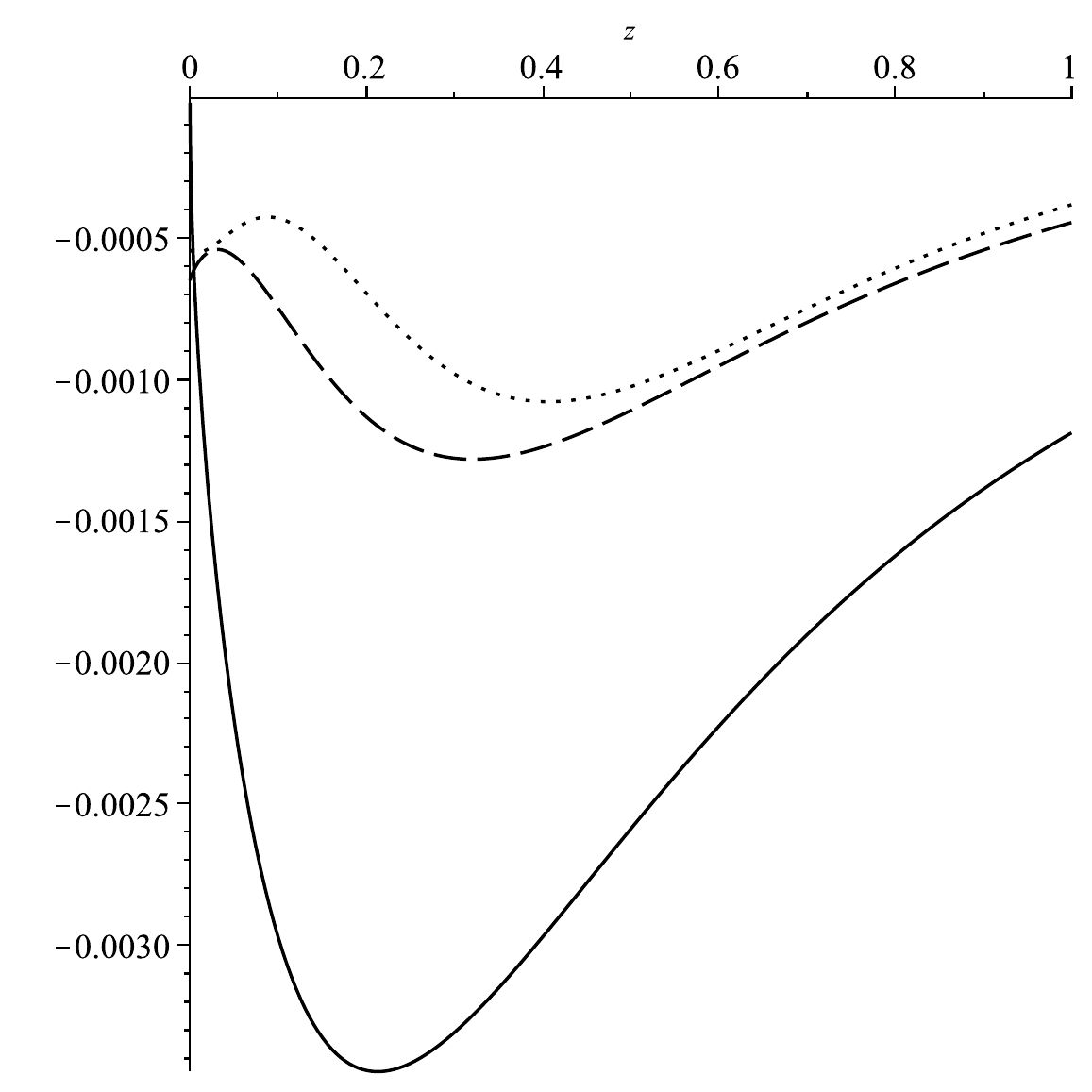}
 \caption{Graphs of $\hat{\kappa}_{s}(z)/z$ for $s=2$ (solid), $s=3$ (dashed) and $s=4$ (dotted) for $0 \leq z \leq 1$.}
 \label{fig:kappa2-4}
\end{figure}

Asymptotic expansions for the positive zeros of the companion function $L_{i\nu}(t)$ (as defined by (\ref{eq71})), which we shall label by $l_{\nu,m}$, come in a similar manner using
\begin{equation}
\label{eq114a}
\frac{2 \sinh(\nu \pi)}{\pi} L_{i\nu}(\nu z)
= \mathcal{Y}(\nu,z)
\mathrm{Bi}\left( \nu^{2/3}
\mathcal{Z}(\nu,z)\right),
\end{equation}
which follows from (\ref{eq62}) and (\ref{eq112}). From this we find that they satisfy the similar expansions as (\ref{eq122}), namely
\begin{equation}
\label{eq122a}
l_{\nu,m} \sim
\nu \sum_{s=0}^{\infty}
\frac{\varrho_{m,s}}{\nu^{2s}}
\quad (\nu \to \infty, \, m=1,2,3,\ldots).
\end{equation}
We have $\varrho_{m,j}=\hat{\kappa}_{j}(\varrho_{m,0})$ with the same functions $\hat{\kappa}_{j}(z)$ as above, but in place of (\ref{eq123}) $\varrho_{m,0}$ satisfies
\begin{equation}
\label{eq123a}
\ln\left\{
\frac{1+(1-\varrho_{m,0}^{2})^{1/2}}{\varrho_{m,0}}
\right\}-\left(1-\varrho_{m,0}^{2}\right)^{1/2}
=\frac{2}{3\nu}|\mathrm{b}_{m}|^{3/2},
\end{equation}
where $\mathrm{b}_{m}$ is the $m$th negative zero of $\mathrm{Bi}(x)$. 

\begin{figure}
 \centering
 \includegraphics[
 width=0.7\textwidth,keepaspectratio]{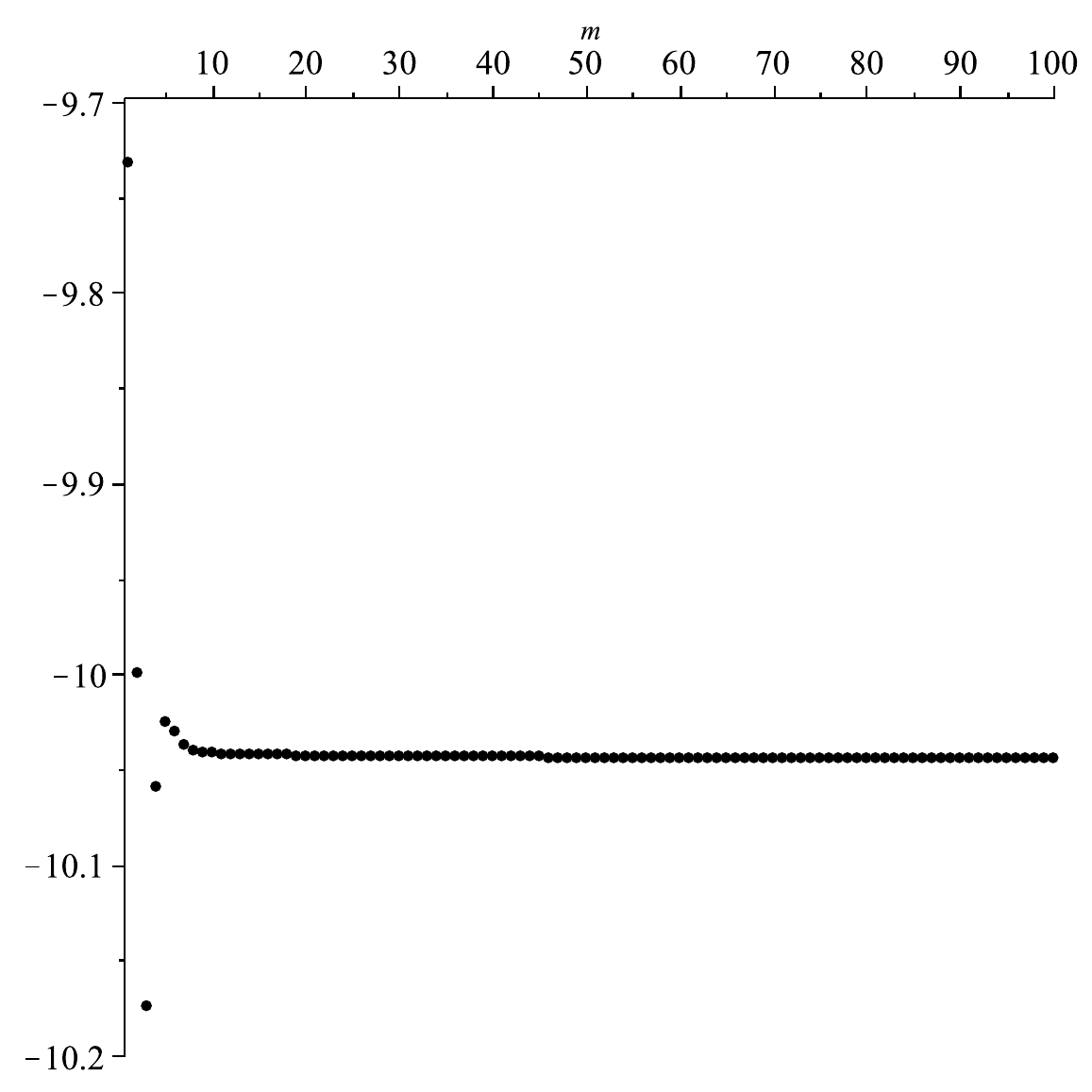}
 \caption{Plots of $\log_{10}|\Delta_{5}(k_{4}(5,m))|$ for $m=1,2,3,\ldots 100$.}
 \label{fig:Airy5zeros}
\end{figure}

\begin{figure}
 \centering
 \includegraphics[
 width=0.7\textwidth,keepaspectratio]{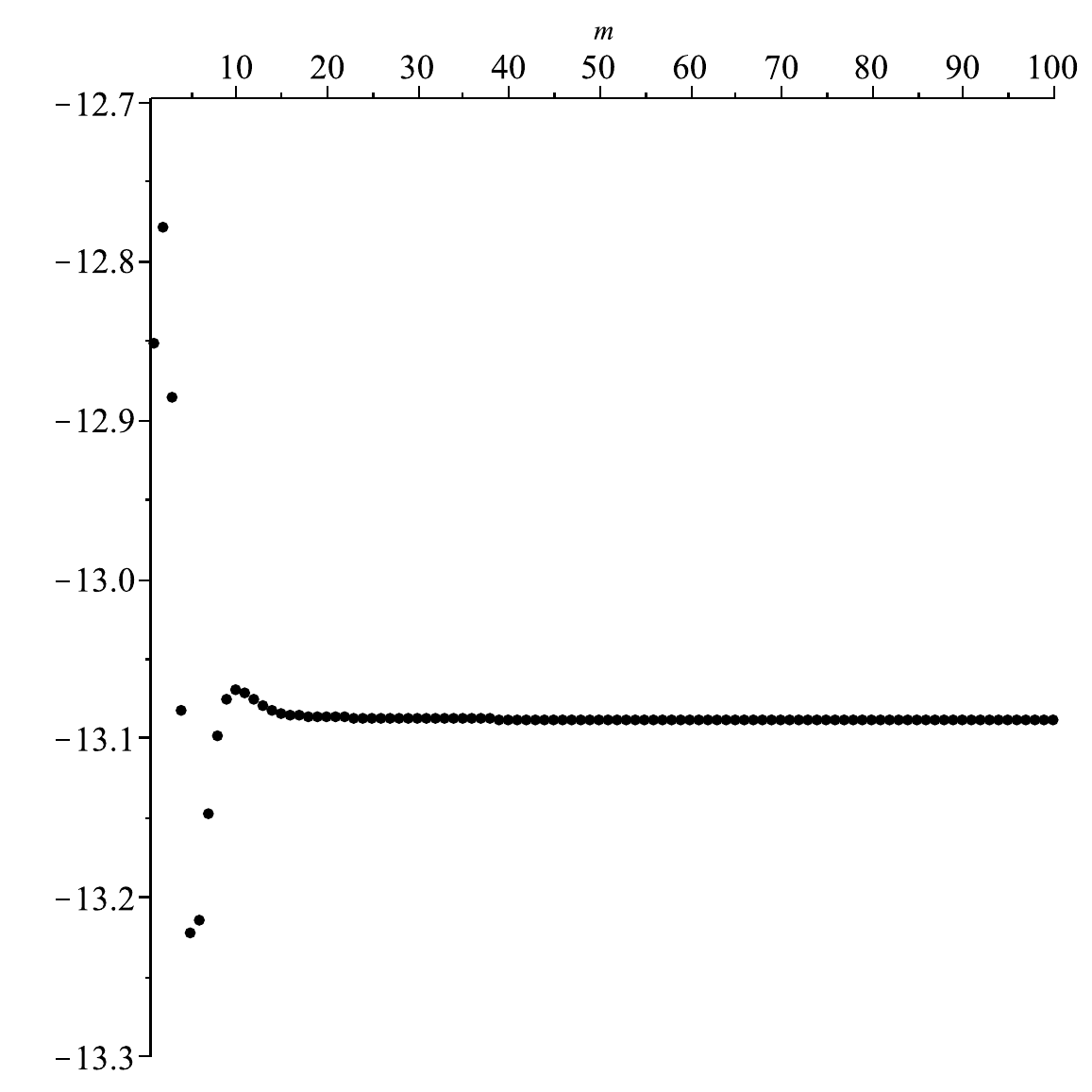}
 \caption{Plots of $\log_{10}|\Delta_{10}(k_{4}(10,m))|$ for $m=1,2,3,\ldots 100$.}
 \label{fig:Airy10zeros}
\end{figure}

\begin{figure}
 \centering
 \includegraphics[
 width=0.7\textwidth,keepaspectratio]{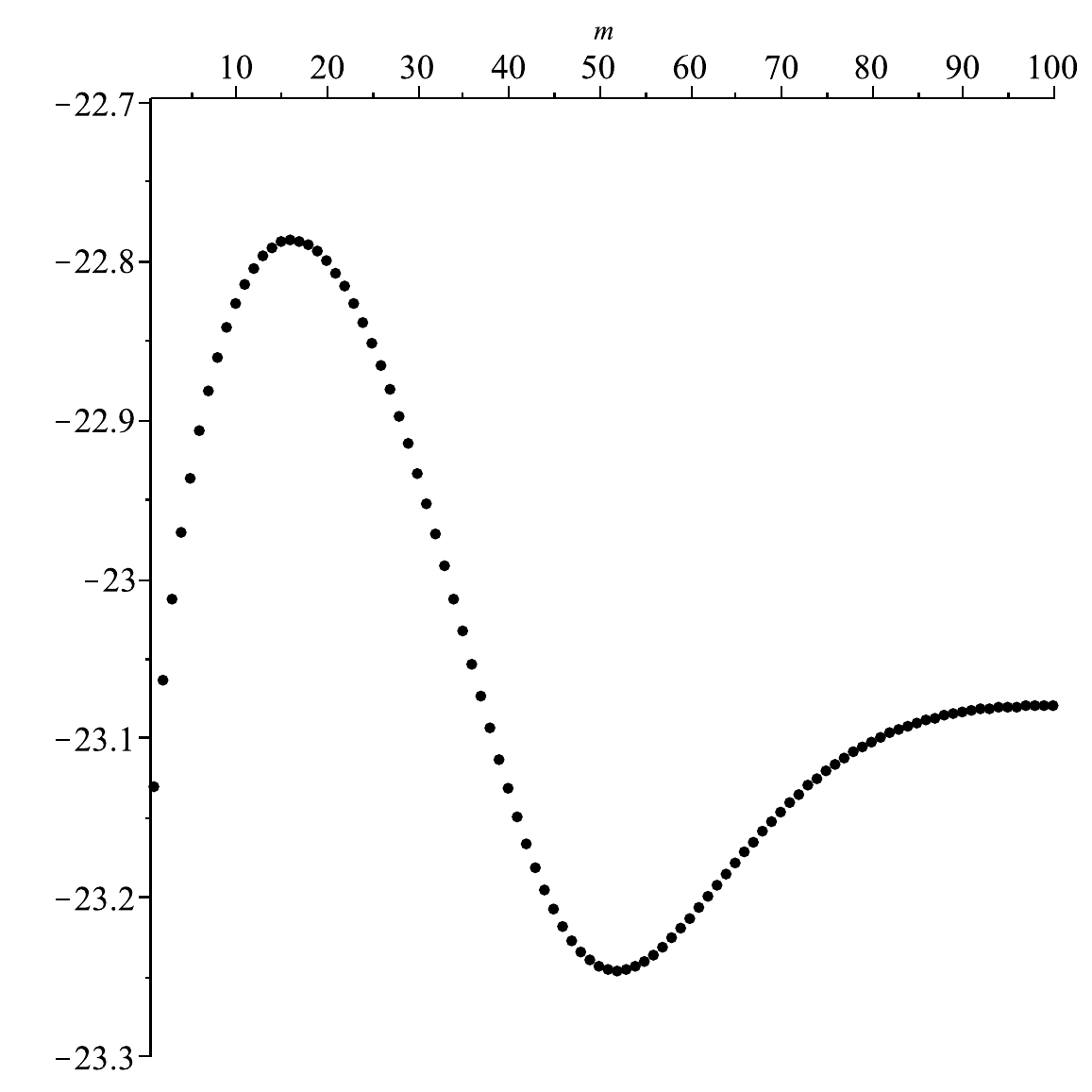}
 \caption{Plots of $\log_{10}|\Delta_{100}(k_{4}(100,m))|$ for $m=1,2,3,\ldots 100$.}
 \label{fig:Airy100zeros}
\end{figure}

We now test the accuracy of our approximations (\ref{eq122}) for the zeros of $K_{i \nu}(t)$. To this end, let
\begin{equation}
\label{eq129}
k_{4}(\nu,m) =
\nu \sum_{s=0}^{4}
\frac{\kappa_{m,s}}{\nu^{2s}}.
\end{equation}
Similarly to (\ref{eq51}) we estimate the relative error $\delta_{\nu,m}$, as defined by
\begin{equation}
\label{eq132}
k_{4}(\nu,m)=\nu k_{\nu,m}\left(1+\delta_{\nu,m}\right),
\end{equation}
and do so by computing $\Delta_{\nu}(k_{4}(\nu,m))$ for various values of $\nu$ and $m$; this is given by (cf. (\ref{eq52}))
\begin{equation}
\label{eq130}
\Delta_{\nu}(t)
= \frac{K_{i\nu}(t)}{t K'_{i\nu}(t)}
=-\frac{K_{i\nu}(t)}{t\Re\{K_{1+i\nu}(t)\}}.
\end{equation}
Here the second equality follows from \cite[Eq. 10.29.2]{NIST:DLMF} to aid in computation. Again by Taylor's theorem we have $\Delta_{\nu}(k_{4}(\nu,m)) = \delta_{\nu,m} + \mathcal{O}( \delta_{\nu,m}^{2})$. For example, for $\nu=10$ and $m=20$ we obtain $k_{4}(10,20)=0.0014850135 \cdots$ and then
\begin{equation}
\label{eq135}
\Delta_{10}(k_{4}(10,20)) 
= \mathbf{8.18131294761}070\cdots\times 10^{-14},
\end{equation}
with the digits in bold agreeing with the exact relative error. The latter value was found in Maple with Digits set to 30 and numerically solving for small $\epsilon$ the equation
\begin{equation}
\label{eq136}
K_{10i}\left(k_{4}(10,20)(1+\epsilon)^{-1}\right) = 0.
\end{equation}

Graphs of the values of $\log_{10}|\Delta_{\nu}(k_{4}(\nu,m))|$ for $m=1,2,3,\ldots 100$ are shown in \cref{fig:Airy5zeros,fig:Airy10zeros,fig:Airy100zeros} for $\nu=5$, $\nu=10$ and $\nu=100$, respectively. We observe uniform high accuracy for all values of $m$.

\section*{Acknowledgments}
I thank the anonymous referees for helpful comments. Financial support from Ministerio de Ciencia e Innovación pro\-ject PID2021-127252NB-I00 (MCIN/AEI/10.13039/ 501100011033/FEDER, UE) is acknowledged.

\section*{Disclosure Statement}
The author has no conflict of interest to declare that is relevant to the content of this article.

\makeatletter
\interlinepenalty=10000

\bibliographystyle{siamplain}
\bibliography{biblio}
\end{document}